\documentclass[a4paper, 12pt]{amsart}
\usepackage[utf8]{inputenc}
\usepackage[margin=1in]{geometry} 
\usepackage{amsmath,amsthm,amssymb,amsfonts,amscd,tikz,tikz-cd,bbm}
\usepackage{mathtools}
\usepackage{mathrsfs}
\usepackage{indentfirst}
\usepackage{enumitem}
\usepackage{graphicx}
\usepackage{subcaption}
\usepackage{float}
\usepackage{array}
\setlist[enumerate,1]{label={(\alph*)}}
\title{Toric supervarieties with one odd dimension}
\author{Eric Jankowski}
\date{}

\newcommand{\A}{\mathscr{A}}
\newcommand{\B}{\mathscr{B}}
\newcommand{\C}{\mathscr{C}}
\newcommand{\F}{\mathscr{F}}

\newcommand{\I}{\mathscr{I}}
\newcommand{\J}{\mathscr{J}}

\renewcommand{\O}{\mathscr{O}}
\renewcommand{\S}{\mathscr{S}}

\renewcommand{\AA}{\mathbb{A}}
\newcommand{\PP}{\mathbb{P}}
\newcommand{\NN}{\mathbb{N}}
\newcommand{\ZZ}{\mathbb{Z}}
\newcommand{\QQ}{\mathbb{Q}}
\newcommand{\RR}{\mathbb{R}}

\newcommand{\kk}{\mathbb{C}}

\newcommand{\tab}{\hspace{0.6cm}}

\renewcommand{\epsilon}{\varepsilon}
\newcommand{\m}{\mathfrak{m}}
\newcommand{\p}{\mathfrak{p}}
\newcommand{\h}{\mathfrak{h}}
\newcommand{\g}{\mathfrak{g}}
\renewcommand{\t}{\mathfrak{t}}
\renewcommand{\d}[1]{\frac{\partial}{\partial #1}}

\renewcommand{\tilde}{\widetilde}

\DeclareMathOperator{\Hom}{Hom}
\DeclareMathOperator{\Spec}{Spec}
\DeclareMathOperator{\Proj}{Proj}

\let\Im\undefined
\DeclareMathOperator{\Im}{Im}

\DeclareMathOperator{\Lie}{Lie}
\DeclareMathOperator{\Gr}{Gr}
\DeclareMathOperator{\GL}{GL}

\DeclareMathOperator{\Stab}{Stab}
\DeclareMathOperator{\Orb}{Orb}
\DeclareMathOperator{\Res}{Res}
\DeclareMathOperator{\Sat}{Sat}

\DeclareMathOperator{\Span}{Span}

\newtheorem{theorem}{Theorem}
\newtheorem{lemma}{Lemma}[section]
\newtheorem{corollary}[lemma]{Corollary}
\newtheorem{proposition}[lemma]{Proposition}

\theoremstyle{definition}
\newtheorem{example}[lemma]{Example}
\newtheorem{remark}[lemma]{Remark}
\newtheorem{definition}[lemma]{Definition}

\begin{document}

\begin{abstract}
    In this paper we describe the notion of a toric supervariety, generalizing that of a toric variety from the classical setting. We give a combinatorial interpretation of the category of quasinormal toric supervarieties with one odd dimension using decorated polyhedral fans. We then use this interpretation to calculate some invariants of these supervarieties and extract geometric information from them.
\end{abstract}


\maketitle

\section{Introduction}
Toric varieties form a rich and interesting class of well-behaved examples in algebraic geometry. Many important varieties are toric: affine spaces, projective spaces, Hirzebruch surfaces, and countless more. Toric varieties (especially normal ones) can be described by purely combinatorial data, a feature that allows one to quickly understand important information about a given variety without significant computation. As a result, they provide a concise dictionary between algebraic geometry and polyhedral combinatorics.

It is natural to ask to what extent this story generalizes to the super setting. Of course, one first requires the notion of an algebraic supertorus. A suitable idea arises in \cite{GSS}, where it is said that a Lie superalgebra $\t$ is \textbf{quasitoral} if $[\t, \t_0]=0$. Maximal quasitoral subalgebras play the same role in the theory of Lie superalgebras that maximal toral subalgebras play in the ordinary setting. Following this analogy, an algebraic supertorus will be a supergroup whose underlying variety is a torus and whose Lie superalgebra is quasitoral.

This definition allows us to consider supervarieties with open dense orbits of supertori, which we find exhibit many similarities to classical toric varieties. One of the most important results in the classical theory gives an equivalence of categories between normal toric varieties and polyhedral fans. It is seen in this paper that a similar equivalence, using decorated fans, can be formulated in the super case with one odd dimension.

\subsection{Structure of paper}
In Section \ref{TVSection}, we describe the combinatorial approach to toric varieties, mainly following \cite{CLS}. In Section \ref{SupergeometrySection}, we outline the fundamentals of algebraic supergeometry, following \cite{CCF}. In Section \ref{SupertoriSection}, we study algebraic supertori and their classifying data.

In Sections \ref{ATSVSection} and \ref{QuasinormalSection}, we generalize much of the classical theory of toric varieties to the supergeometric setting with one odd dimension. This treatment includes several equivalent constructions of affine toric supervarieties (with one odd dimension) as well as a combinatorial description of the category of quasinormal toric supervarieties (with one odd dimension). Moreover, the Duflo-Serganova functor of \cite{DS} is used to provide an interesting invariant of these toric supervarieties.

As an application of these combinatorics, in Section \ref{GeometrySection}, we give criteria for splitness and smoothness of toric supervarieties. We also characterize orbits and stabilizers of points in toric supervarieties, and determine the conditions under which fiber products exist in the category of quasinormal toric supervarieties with one odd dimension.

\subsection{Acknowledgements}
The author is grateful to his advisor, Vera Serganova, for inspiring his interest in algebraic supergeometry and for many useful discussions about toric supervarieties. This material is based upon work supported by the National Science Foundation Graduate Research Fellowship Program under Grant No.\ 2146752.

\section{Toric varieties}\label{TVSection}

Our ground field is $\kk$, although we could just as well work over any algebraically closed field of characteristic 0.

\subsection{Algebraic tori}

\begin{definition}
    An \textbf{algebraic torus} is an algebraic group $T$ isomorphic to some $(\kk^\times)^n$, where $\kk^\times \cong \Spec \kk[x^{\pm 1}]$ is the multiplicative group of the complex numbers.
\end{definition}

A torus $T$ has both a \textbf{character lattice} $M = \Hom(T, \kk^\times)$ and a \textbf{cocharacter lattice} $N=\Hom(\kk^\times,T)$. These lattices are dual to each other, admitting a nondegenerate perfect pairing
\begin{align*}
    \langle\,,\rangle : M \times N \to \ZZ
\end{align*}
given by composing the morphisms of algebraic groups $\kk^\times \to T \to \kk^\times$. Given an arbitrary lattice $N$, one may define the torus $T_N := N \otimes_\ZZ \kk^\times$ and observe that $T_N$ has cocharacter lattice $N$ and character lattice $M=N^*$. The character lattice $M$ is naturally a basis for the regular functions on $T_N$, in the sense that the coordinate ring satisfies $\kk[T_N] = \kk[M]$.

If we fix an identification $N \cong \ZZ^n$, then we likewise obtain $M \cong \ZZ^n$ and $T_N \cong (\kk^\times)^n$, so that $\kk[T_N] \cong \kk[x_1^{\pm 1}, ..., x_n^{\pm 1}]$.

\subsection{Affine toric varieties}\label{ATVBackground}

\begin{definition}
    A \textbf{toric variety} is an (irreducible) algebraic variety $X$ containing a torus $T$ as a Zariski open subset such that the action of $T$ on itself extends to an algebraic action of $T$ on the whole variety.
\end{definition}

An affine toric variety may be constructed in a number of equivalent ways. Firstly, given a finite set $\A = \{m_1, ..., m_r\}$ in a lattice $M'$, one can define $N' = (M')^*$ and subsequently a morphism of varieties
\begin{align*}
    \Phi_{\A} : T_{N'} &\to \kk^r
\end{align*}
given by $\Phi_{\A}(t) = (t^{m_1}, ..., t^{m_r})$. The Zariski closure of the image $\Phi_{\A}(T_{N'}) \subseteq \kk^r$ becomes an affine toric variety $Y_{\A}$ with torus $\Phi_{\A}(T_{N'}) \cong T_{(\ZZ \A)^*}$.

Such subvarieties of $\kk^r$ may instead be characterized by their \textbf{toric ideals}, which are prime ideals of the form $$(x^{\ell_+} - x^{\ell_-} \mid \ell_+, \ell_- \in \NN^s \text{ and } \ell_+ - \ell_- \in L)$$ for some sublattice $L \subseteq \ZZ^r$. Equivalence with the prior construction may be established by viewing $\Phi_{\A}$ as a map $T_{N'} \to (\kk^\times)^r$ and defining $L$ as the kernel of the induced map $\bar \Phi_{\A}^* : \ZZ^r \to M$ on character lattices.

Equivalently, one may begin with an \textbf{affine semigroup}, i.e.\ a finitely generated semigroup $S$ that can be embedded in a lattice. The semigroup algebra $\kk[S]$ then becomes the coordinate ring of the toric variety $\Spec \kk[S]$ whose torus is $T_{(\ZZ S)^*}$. If $S = \NN \A$, then $\Spec \kk[S] \cong Y_{\A}$.

We observe that $T_N$ acts on monomials $x^m$ in its coordinate ring $\kk[T_N] = \kk[M]$ by the corresponding characters $m \in M$. Moreover, if $\Spec \kk[S]$ is an affine toric variety with torus $T_N$, then $\kk[S]$ is a $T_N$-invariant subalgebra of $\kk[M]$.

A \textbf{toric morphism} between toric varieties is one which restricts to a group homomorphism on their underlying tori. A toric morphism between affine toric varieties $\phi : \Spec \kk[S] \to \Spec \kk[S']$ is characterized by the fact that its dual map $\phi^* : \kk[S'] \to \kk[S]$ arises from a map of semigroups $S' \to S$.

\subsection{Normal toric varieties}\label{NTVBackground}

\begin{definition}
    An algebraic variety is \textbf{normal} if the local ring at every point is an integrally closed domain.
\end{definition}

Normal toric varieties can be described especially easily with combinatorial data. In order to obtain a normal affine toric variety $\Spec \kk[S]$, one requires that the semigroup $S$ be \textbf{saturated}, meaning that if $km \in S$ for $m \in \ZZ S$ and $k \geq 0$, then $s \in S$.

Such semigroups can also be characterized in terms of cones. Let $N$ be a lattice in a real vector space $N_\RR := N \otimes_\ZZ \RR$. A \textbf{cone} $\sigma$ in $N_\RR$ is a nonempty subset closed under addition and multiplication by nonnegative scalars. We say $\sigma$ is \textbf{polyhedral} if it is generated (under these operations) by finitely many points, and \textbf{rational} if these points can be taken to lie in $N$. Finally, $\sigma$ is \textbf{strongly convex} if it contains no positive-dimensional subspace of $N_\RR$.

Let $M = N^*$ be the dual lattice of $N$, so $M_\RR := M \otimes_\ZZ \RR$ is the dual vector space of $N_\RR$. The \textbf{dual cone} of $\sigma$ is the cone $\check \sigma = \{m \in M_\RR \mid \langle m, u \rangle \geq 0 \text{ for all } u \in \sigma\}$ in $M_\RR$. If $\sigma$ is a strongly convex rational polyhedral cone, then the lattice points $\check \sigma \cap M$ of $\check \sigma$ form a saturated affine semigroup called $S_\sigma$. This construction leads to the classic characterization of normal affine toric varieties.

\begin{proposition}
    Let $X$ be an affine toric variety with torus $T_N$. Then the following are equivalent:
    \begin{enumerate}
        \item $X$ is normal.
        \item $X \cong \Spec \kk[S]$ for a saturated affine semigroup $S \subseteq M = N^*$ such that $\ZZ S = M$.
        \item $X \cong X_\sigma := \Spec \kk [S_\sigma]$ for a strongly convex rational polyhedral cone $\sigma \subset N_\RR$.
    \end{enumerate}
\end{proposition}

The language of cones is a powerful one, as it allows us to translate results about toric varieties into results about convex geometry, and vice-versa. Moreover, it allows for a nice characterization of arbitrary normal toric varieties via polyhedral fans:

A \textbf{face} of the cone $\sigma$ is a sub-cone $H_m \cap \sigma$ of $\sigma$, where $H_m := \ker m \cap \sigma = \{u \in \sigma \mid \langle m, u \rangle = 0\}$ for some $m \in \check \sigma$. A \textbf{fan} $\Sigma$ in $N_\RR$ is a finite collection of strongly convex rational polyhedral cones $\sigma$ such that $\Sigma$ is closed under taking faces, and the intersection of any two cones in $\Sigma$ is a face of each.

One may construct a toric variety $X_\Sigma$ from a fan $\Sigma$ by gluing together the open affines $X_\sigma$ for $\sigma \in \Sigma$ along their intersections (which correspond to intersections of cones). A theorem of Sumihiro \cite{Sumihiro} asserts that all normal toric varieties arise in this way.

Toric morphisms between normal toric varieties can likewise be described entirely in terms of their corresponding fans. In particular, a morphism $\phi : X_\Sigma \to X_{\Sigma'}$ is toric if and only if it arises from a lattice map $\bar \phi : N \to N'$ that is compatible with the fans in the sense that for every $\sigma \in \Sigma$, there is $\sigma' \in \Sigma'$ such that $\bar \phi_\RR(\sigma) \subseteq \sigma'$. In this sense, the category of normal toric varieties is equivalent to the category of fans.

\section{Algebraic supergeometry}\label{SupergeometrySection}

We briefly introduce the basic notions of complex algebraic supergeometry that will be relevant to this paper.

\subsection{Superalgebras}
\begin{definition}
A \textbf{superalgebra} is a $\ZZ/2\ZZ$-graded, unital, associative $\kk$-algebra $A = A_0 \oplus A_1$.
\end{definition}

If $a \in A$ is homogeneous, we denote by $|a| \in \{0,1\}$ the \textbf{parity} of $a$. For the purposes of this paper, all superalgebras $A$ will be \textbf{commutative} in the sense that the multiplication satisfies $a \cdot b = (-1)^{|a| \cdot |b|}b \cdot a$ for homogeneous elements $a \in A_{|a|}$ and $b \in A_{|b|}$. We will generally use Roman letters to denote even elements and Greek letters to denote odd elements; e.g.\ $\kk[x_1, x_2, \xi_1, \xi_2, \xi_3]$ is a superalgebra with two even generators and three odd generators.

\textbf{Morphisms of superalgebras} will always respect parity. Note that if $A$ and $B$ are superalgebras, then so is $A \otimes_\kk B$:
    \begin{align*}
        (A \otimes_\kk B)_0 &= (A_0 \otimes_\kk B_0) \oplus (A_1 \otimes_\kk B_1) \\
        (A \otimes_\kk B)_1 &= (A_0 \otimes_\kk B_1) \oplus (A_1 \otimes_\kk B_0)
    \end{align*}

Let $A$ be a commutative superalgebra. We say $A$ is an \textbf{integral superdomain} if the ideal $(A_1)$ generated by $A_1$ is a prime ideal containing all the zero divisors of $A$. It is \textbf{local} if it has a unique maximal ideal. As in the usual setting, the localization of a superalgebra at a prime ideal is local.

A \textbf{Hopf superalgebra} is a superalgebra $A$ which is simultaneously a super coalgebra and is equipped with a map called the antipode, satisfying all the usual diagrams.

\subsection{Superspaces}

\begin{definition}
    A \textbf{superspace} $X = (|X|, \O_X)$ is a pair consisting of a topological space $|X|$ and a sheaf $\O_X$ of superalgebras such that each stalk $\O_{X,x}$ is a local superalgebra. A \textbf{morphism of superspaces} is a pair $\phi = (|\phi|, \phi^\#) : X \to Y$ for a continuous map $|\phi| : |X| \to |Y|$ and a morphism of $\phi^\# : \O_Y \to |\phi|_* \O_X$ of sheaves on $|Y|$ which becomes a local morphism of local superalgebras on each stalk.
    
    We say that $X$ is a \textbf{superscheme} if the odd part $\O_{X,1}$ of the structure sheaf is a quasi-coherent sheaf of $\O_{X,0}$-modules. An \textbf{open subscheme} of $X$ is given by the restriction of $\O_X$ to an open set $U \subseteq |X|$, and a \textbf{closed subscheme} is an equivalence class of closed immersions in the usual sense. If $|X|$ is Noetherian and irreducible, we write $\kk(X)$ for the stalk of $\O_X$ at the generic point.
\end{definition}

Let $(\O_{X,1})$ be the ideal sheaf in $\O_X$ generated by the odd elements at every open $U \subseteq |X|$. Let $\O_{\overline{X}}$ be the sheaf of ordinary algebras defined by $\O_{\overline{X}}(U) = (\O_X/(\O_{X,1}))(U)$. We call $\overline{X} := (|X|, \O_{X}^r)$ the \textbf{bosonic reduction} of $X$.

Given a superalgebra $A = A_0 \oplus A_1$, we observe that $A$ is an $A_0$-module via left multiplication, allowing us to construct the \textbf{affine superscheme} $\Spec A = (|\Spec A_0|, \tilde A)$ where $\tilde A$ is the sheaf associated to the module $A$. Letting $\overline{A} = A/(A_1)$, we note that $|\Spec A_0| = |\Spec \overline{A}|$ since the two algebras differ only by nilpotents. Hence $\overline{\Spec A} \cong \Spec \overline{A}$. It can be shown that affine superschemes are the local models of arbitrary superschemes (\cite{CCF}, Proposition 10.1.3).

The following definition is due to Sherman \cite{Sherman}.

\begin{definition}\label{SupervarietyDefinition}
    An \textbf{algebraic supervariety} is a superscheme $X$ over $\kk$ such that all of the following conditions hold:
    \begin{enumerate}[label=(\roman*)]
        \item $X$ is irreducible (i.e.\ $|X|$ is an irreducible topological space).
        \item $X$ is separated (i.e.\ the diagonal morphism $X \to X \times_\kk X$ is a closed embedding).
        \item $X$ is of finite type (i.e.\ $X$ admits a cover by finitely many open affine subschemes $\Spec A$ where the $A$ are finitely generated superalgebras over $\kk$).
        \item For any open subscheme $U \subseteq X$, the restriction map $\O_X(U) \to \kk(X)$ is injective.
        \item $\kk(X)$ is an integral superdomain.
    \end{enumerate}
\end{definition}

The latter two conditions provide a weakened version of the usual integrality condition for ordinary algebraic varieties. An algebraic supervariety is moreover \textbf{integral} if (v) is strengthened to
\begin{enumerate}[label=(v$'$)]
        \item For any open subscheme $U \subseteq X$, $\O_X(U)$ is an integral superdomain.
\end{enumerate}

Unless a supervariety $X$ is integral, it is not true that its bosonic reduction $\overline{X}$ is an algebraic variety (since integrality will fail). However, it is still possible to extract an ordinary algebraic variety from an arbitrary supervariety: Define an ideal sheaf $\I$ in $\O_X$ by setting $\I(U)$ equal to the preimage of $(\kk(X)_1)$ under the restriction map $\O_X(U) \to \kk(X)$. Then $\I$ contains $(\O_{X,1})$ by part (iv) of the above definition, so we may define the \textbf{underlying variety} $X_0:= (|X|, \O_X/\I)$. To see that it is a variety, we observe that the quotient kills all zero divisors in $\O_X(U)$ by part (v) of the above definition. One may alternatively construct $X_0$ as the reduced induced closed subscheme of $\overline{X}$ with the same topological space $|X|$.

\begin{example}
    Consider the subalgebra $A = \kk[x, \xi_1, x\xi_2, \xi_1\xi_2] \subseteq \kk[x, \xi_1, \xi_2]$. One can check that $X=\Spec A$ is a non-integral supervariety with $\overline{X} \cong \Spec \kk[x, \xi_1\xi_2]$ and $X_0 \cong \Spec \kk[x]$.
\end{example}

\begin{proposition}\label{Affine}
Let $X$ be a supervariety. The following are equivalent:
\begin{enumerate}
    \item $X$ is affine.
    \item $\overline{X}$ is affine.
    \item $X_0$ is affine.
\end{enumerate}
\end{proposition}
\begin{proof}
    See \cite{Zubkov} for (a) $\iff$ (b) and \cite{Hartshorne}, Exercise III.3.1 for (b) $\iff$ (c).
\end{proof}

\begin{proposition}\label{IntegralSupervarietyProp}
Let $X$ be a supervariety. The following are equivalent:
\begin{enumerate}
    \item $X$ is integral.
    \item $\overline{X}$ is integral.
    \item $\overline{X} = X_0$.
\end{enumerate}
\end{proposition}
\begin{proof}
For (a) $\implies$ (c), suppose $X$ is integral. It suffices to show that $\I = (\O_{X,1})$. The inclusion $(\O_{X,1}) \subseteq \I$ is already known. Let $\xi \in \I(U)$, so $\xi^2 = 0$. Since $\O_X(U)$ is integral, we have $\xi \in (\O_{X,1})(U)$, and so $\overline{X} = X_0$. Conversely, if $\overline{X} = X_0$ (both of which are closed subschemes of $X$), then by the correspondence of closed subschemes with ideal sheaves, we obtain $\I = (\O_{X,1})$ and hence $X$ is integral. Therefore (a) $\iff$ (c).

We also have (b) $\iff$ (c) because $X_0$ is the reduced induced closed subscheme structure on $\overline{X}$.
\end{proof}

The following definition occurs in \cite{Sherman}, Section 3.3.2.

\begin{definition}
    Given an algebraic supervariety $X$, we may find a dense open subscheme $U \subseteq X$ such that its bosonic reduction $\overline{U}$ is smooth. The conormal sheaf $(\O_{X,1})/(\O_{X,1})^2$ is coherent on $U$, so it trivializes over a dense open subscheme $\overline{U}' \subseteq \overline{U}$. We then define the \textbf{dimension} of $X$ to be the dimension of $\m_x / \m_x^2$ as a super vector space for any choice of $x \in \overline{U}'$. In particular, if $\dim \overline{U} = n$ and the rank of the vector bundle over $\overline{U}'$ is $r$, then $\dim X = (n|r)$.
\end{definition}

The following result shows that for the purposes of this paper, the distinction between $\overline{X}$ and $X_0$ is insignificant.
\begin{proposition}\label{IntegralIfOneOddDim}
    An $(n|1)$-dimensional supervariety is integral.
\end{proposition}
\begin{proof}
    It suffices to show for an affine supervariety $X$. We have an injection $\kk[X] \to \kk(X)$ where $\kk(X)_1^2 = 0$. Hence $(\kk(X)_1) = \kk(X)_1$, so all zero divisors in $\kk(X)$ (and thus all zero divisors in $\kk[X]$) are odd. Therefore $\kk[X]$ is integral, and we are finished.
\end{proof}

We will later encounter definitions of graded, split, and smooth supervarieties in Sections \ref{SplitSubsection}-\ref{SmoothSubsection}.

\subsection{Algebraic supergroups}

\begin{definition}
    An \textbf{affine algebraic supergroup} is any of the following equivalent notions:
    \begin{enumerate}
        \item The spectrum of a finitely-generated Hopf superalgebra $R$ which is an integral superdomain
        
        \item A group object $G$ in the category of affine supervarieties
        
        \item An affine supervariety $G = \Spec R$ whose functor of points $G(A) := \Hom(R,A)$ is group-valued
        
        \item A pair $G = (G_0, \g)$ where $G_0$ is an ordinary (irreducible) affine algebraic group, $\g = \g_0 \oplus \g_1$ is a finite-dimensional Lie superalgebra such that $\Lie G_0 = \g_0$, and $\g$ is equipped with a $G_0$-module structure such that the corresponding $\g_0$-module is adjoint (called a \textbf{Harish-Chandra pair})
    \end{enumerate}
\end{definition}

We use these notions (which can be found in \cite{CCF}, Sections 7 and 11) interchangeably throughout this paper. The requirement that the Hopf superalgebra $R$ of part (a) be integral is needed only to guarantee irreducibility of the supergroup. If we relax the assumption that a supervariety must be integral (and accordingly adjust for the possibility of multiple generic points), then this requirement may be dropped, and definitions (a)-(d) remain equivalent.

The notation in the definition of a Harish-Chandra pair is intentional, as $G_0 = \overline{G}$ is the underlying algebraic group of $G$ and $\g$ is the Lie superalgebra (i.e.\ tangent space at the identity) of $G$. To see the equivalence of definition (a) with definitions (b) and (c), we note the following:
\begin{proposition}
    An affine algebraic supergroup is an integral supervariety.
\end{proposition}
\begin{proof}
    Using definition (b), consider the bosonic reduction $\overline{G}$, which is an affine group scheme satisfying all properties of an algebraic variety, except possibly reducedness. But any group scheme over a field of characteristic 0 is reduced (\cite{Perrin}, I, Theorem 1.1 and I, Corollary 3.9 and II, Theorem 2.4), so indeed $\overline{G}$ is integral. We are therefore finished by Proposition \ref{IntegralSupervarietyProp}.
\end{proof}

An \textbf{action} of an algebraic supergroup on a supervariety is defined as a morphism $$\rho : G \times X \to X$$ satisfying the usual properties.

\section{Algebraic supertori}\label{SupertoriSection}

\begin{definition}
    An \textbf{algebraic supertorus} is an algebraic supergroup $T = (T_0, \t)$ such that $T_0 \cong (\kk^\times)^n$ is an ordinary algebraic torus lying in the center of $T$.
\end{definition}

This definition is motivated by the theory of Cartan subalgebras of quasireductive Lie superalgebras as in \cite{SerganovaQred}. A \textbf{quasireductive} Lie superalgebra $\g$ is one whose even part $\g_0$ is reductive. In this setting, one chooses a Cartan subalgebra $\h_0$ of $\g_0$, and defines $\h$ as the centralizer of $\h_0$ in $\g$. Thus, an algebraic supertorus is an algebraic supergroup $T$ such that $T_0$ is a torus and $\Lie T$ is a Cartan subalgebra of a quasireductive Lie superalgebra. The representation theory of these superalgebras (and thus of algebraic supertori) has been studied in \cite{GSS}.

\begin{example}
    Define $Q(1)$ as the supergroup whose $A$-points consist of matrices $$\begin{pmatrix}
        a & \xi \\
        \xi & a
    \end{pmatrix}$$ for $a \in A_0^\times$ and $\xi \in A_1$, equipped with the usual matrix multiplication. Then $Q(1)$ is a $(1|1)$-dimensional supertorus.
\end{example}

\subsection{Supergroup structure}

If $\dim T$ = $(n|1)$, then $\t$ is an $(n|1)$-dimensional Lie superalgebra with $\t_0 \subset Z(\t)$ and an odd generator $\theta$ such that $\theta^2 := \frac{1}{2}[\theta, \theta] \in \t_0$. An $(n|1)$-dimensional supertorus is therefore parameterized by a point $c \in \t_0 \cong N_\kk := N \otimes_\ZZ \kk$, where $N$ is the cocharacter lattice of $T_0$.

\begin{definition}
    Denote by $T_N$ the ordinary algebraic torus with character lattice $N$, and $T_{N,c} = (T_N, \t)$ the supertorus with one odd dimension corresponding to the parameter $c = \theta^2 \in N_\kk$.
\end{definition}

Let $M = N^*$ be the character lattice of $T_N$. The supertorus $T_{N,c}$ will have coordinate ring $\kk[T_{N,c}] = \kk[M, \xi]$, with a Hopf superalgebra structure
\begin{align*}
    \Delta x^m &= x^m \otimes x^m (1+ \langle m, c \rangle \xi \otimes \xi) & S x^m &= x^{-m} & \epsilon x^m &= 1 \\
    \Delta \xi &= \xi \otimes 1 + 1 \otimes \xi &
    S \xi &= -\xi & \epsilon \xi &= 0
\end{align*}
where $\langle\, , \rangle$ is the perfect pairing $M \otimes N \to \ZZ$, extended $\kk$-linearly to $M \otimes N_\kk$.

Fixing an identification $N \cong \ZZ^n$ so that $c = (c_1, ..., c_n) \in \kk^n$, write $T_{N,c}(A) = \Hom(\Spec A, T_{N,c})$ for the functor of points of $T$ evaluated on a superalgebra $A$. Then the group law of $T_{N,c}(A)$ is given by
\begin{align*}
    (t_1, ..., t_n \mid \xi) \cdot (u_1, ..., u_n \mid \eta) &= (t_1 u_1 (1+c_1 \xi \eta), ..., t_n u_n (1+c_n \xi \eta) \mid \xi + \eta)
\end{align*}
for $t_i, u_i \in (A_0)^\times$ and $\xi, \eta \in A_1$.

\subsection{Representation theory}

The coordinate ring $\kk[T_{N,c}]$ admits an action by $\t$, wherein the even part acts as usual (fixing $\xi$) and the odd generator acts by $x^m \xi \mapsto x^m$ and $x^m \mapsto \langle m, c \rangle x^m \xi$. Using the above identification $T_0 \cong \ZZ^n$, we may write $\t$ as the subalgebra
$$\Span_\kk \left( x_1 \d{x_1}, ..., x_n \d{x_n} \right) \oplus \Span_\kk \left( \xi \left( c_1 x_1 \d{x_1} + ... + c_nx_n\d{x_n} \right) + \d\xi \right)$$
of the Lie superalgebra of derivations on $\kk[T_{N,c}] \cong \kk[x_1^{\pm 1}, ..., x_n^{\pm 1}, \xi]$.
Thus, as a representation of $T_{N,c}$, $\kk[T_{N,c}]$ splits into a direct sum of $(1|1)$-dimensional representations $L(m) = \Span_\kk ( x^m, x^m \xi )$ for $m \in M$.

\begin{lemma}\label{SubrepsLem}
If $T=T_{N,c}$, then $L(m)$ is irreducible if and only if $\langle m, c \rangle \neq 0$. If $\langle m, c \rangle = 0$, then $L(m)$ is indecomposable but has a $(1|0)$-dimensional subrepresentation $\langle x^m \rangle$.
\end{lemma}

\begin{proof}
The even part preserves the span of each monomial, and the odd part acts as above.
\end{proof}

\subsection{Decomposability}

Under an identification $T_N \cong (\kk^\times)^n$ as above, we obtain $c = (c_1, ..., c_n) \in \kk^n$. Without loss of generality, arrange the indices so that $$c = (c_{1,1}, ..., c_{1,k_1}, ..., c_{r,1}, ..., c_{r,k_r},0,...,0)$$ where for each $i$, $\{c_{i,1}, ..., c_{i,k_i}\}$ is a maximal collection of nonzero rational multiples of an element of $\kk$. That is, we may assume $c_{i,j} = \ell_{i,j} c_i$ for some $c_i \in \kk$ and $\ell_{i,j} \in \ZZ$.

Since the $\GL(n,\ZZ)$-orbit of a point in $\ZZ^n$ consists of all $n$-tuples with the same least common multiple, there exists $g_i \in GL(k_i, \ZZ)$ such that $g(\ell_{i,1}, ..., \ell_{i,k_i}) = (\ell_i, 0, ..., 0)$ for some $\ell_i \in \ZZ$. If $g = (g_1, ..., g_r) \in \GL(k_1,\ZZ) \times ... \times \GL(k_r, \ZZ) \subseteq \GL(n,\ZZ)$, then (up to re-ordering the entries) we obtain $$gc = (c_1, ..., c_r, 0,...,0)$$ for some $\QQ$-linearly independent elements $c_1, ..., c_r \in \kk$. It follows that $$T_{N,c} \cong T_{\ZZ^r, (c_1, ..., c_r)} \times (\kk^\times)^{n-r}.$$

\begin{definition}
    We say a supertorus $T$ (of any dimension) is \textbf{indecomposable} if it cannot be written as a product $T' \times T''$ of two smaller supertori.
\end{definition}

We have seen that if the entries of $c = (c_1, ..., c_n) \in \kk^n$ admit a $\QQ$-linear dependency, then $T_{\ZZ^n, c}$ is decomposable. The following proposition shows that this condition exactly determines indecomposability of a supertorus with one odd dimension.

\begin{proposition}
    Let $T=T_{\ZZ^n,c}$ be a supertorus such that $c_1, ..., c_n \in \kk$ are $\QQ$-linearly independent. Then $T$ is indecomposable.
\end{proposition}
\begin{proof}
    Suppose $T = T' \times T''$. Since $T$ has only one odd dimension, we may assume without loss of generality that $T''$ is a purely even torus of nonzero dimension. Hence $T' = T_{\ZZ^r,c'}$ for some $c' \in \kk^r$, and $T'' = (\kk^\times)^{n-r}$. Then $c = (c_1', ..., c_r', 0, ..., 0)$ is a contradiction which proves the result.
\end{proof}

\subsection{Morphisms}

We will need the following lemmas on morphisms between supertori.

\begin{lemma}\label{ImageLem}
    Let $\phi : T \to T'$ be a morphism of supergroups, where $T$ and $T'$ are both supertori. Then the image of $\phi$ is a supertorus which is closed in $T'$.
\end{lemma}

\begin{proof}
    We may view the morphism as living in the category of Harish-Chandra pairs, so that we have a morphism $\phi_0 : T_0 \to T_0'$ of algebraic groups and a compatible morphism $d\phi : \t \to \t'$ of Lie superalgebras. By the classical fact (\cite{CLS}, Proposition 1.1.1), $\phi_0(T_0)$ is a closed subtorus of $T_0'$. Moreover, $d\phi(\t)$ is a Lie superalgebra whose even part is the Lie algebra of $\phi_0(T_0)$. Since $t_0'$ lies in the center of $t'$, the same holds for $\Lie \phi_0(T_0) \subseteq \t_0'$ in $d\phi(\t) \subseteq \t'$, so the lemma follows.
\end{proof}

\begin{lemma}\label{MorphismOfSupertoriLemma}
    Let $\phi : T_{N,c} \to T_{N',c'}$ be a morphism of supervarieties. Then $\phi$ is a morphism of supergroups if and only if all of the following conditions hold.
    \begin{enumerate}[label=(\roman*)]
        \item There is a lattice map $\bar \phi : N \to N'$ such that $\phi|_{T_N} = \bar \phi \otimes 1_{\kk^\times}$.
        \item The dual map $\phi^* : \kk[M',\xi'] \to \kk[M,\xi]$ satisfies $\phi^*(\xi') = a\xi$ for some $a \in \kk$.
        \item $\bar \phi_\kk = \bar \phi \otimes 1_\kk : N_\kk \to N'_\kk$ satisfies $\bar \phi(c) = a^2 c'$.
    \end{enumerate}
\end{lemma}

\begin{proof}
    Condition (i) is required as in the usual setting. In order for $\phi$ to be a morphism of supergroups, we require $\phi^*$ to be a morphism of Hopf superalgebras. In particular, we require $(\phi^* \otimes \phi^*) \circ \Delta' = \Delta \circ \phi^*$. Since each of these maps is a morphism of superalgebras, it suffices to check the identity on the generators $x^{m'}$ and $\xi'$. First, observe that
    \begin{align*}
        &(\phi^* \otimes \phi^*) \circ \Delta'(\xi') &&\Delta \circ \phi^*(\xi') \\
        &= (\phi^* \otimes \phi^*)(\xi' \otimes 1 + 1 \otimes \xi') &&= \Delta(\phi^*(\xi')) \\
        &= \phi^*(\xi') \otimes 1 + 1 \otimes \phi^*(\xi') &&= \sum_{i=1}^\ell a_i \Delta(x^{k_i} \xi) \\ 
        &= \left( \sum_{i=1}^\ell a_i x^{k_i} \xi \right) \otimes 1 + 1 \otimes \left( \sum_{i=1}^\ell a_i x^{k_i} \xi \right) &&= \sum_{i=1}^\ell a_i x^{k_i} \otimes x^{k_i} (\xi \otimes 1 + 1 \otimes \xi)
    \end{align*}
    where we write $\phi^*(\xi') = \sum_{i=1}^\ell a_i x^{k_i} \xi$ for some $a_i \in \kk$ and $k_i \in M$.
    Since the two sides are equal, it follows that all $k_i = 0$ and hence $\phi^*(\xi') = a \xi$ for some $a \in \kk$. Moreover, for $m' \in M$, we have
    \begin{align*}
        &(\phi^* \otimes \phi^*) \circ \Delta'(x^{m'}) &&\Delta \circ \phi^*(x^{m'}) \\
        &= (\phi^* \otimes \phi^*)(x^{m'} \otimes x^{m'}(1+ \langle m', c' \rangle) \xi' \otimes \xi') &&= \Delta(x^{\bar \phi^*(m')}) \\
        &= x^{\bar \phi^*(m')} \otimes x^{\bar \phi^*(m')} (1+\langle m', c' \rangle \phi^*(\xi') \otimes \phi^*(\xi')) &&= x^{\bar \phi^*(m')} \otimes x^{\bar \phi^*(m')} (1 + \langle \bar \phi^*(m'), c \rangle \xi \otimes \xi)
    \end{align*}
    so that $a^2 \langle m', c' \rangle = \langle \bar \phi^*(m'), c \rangle$ for all $m' \in M'$. Since $\langle\,,\rangle$ is a nondegenerate perfect pairing, we know $\langle \bar \phi^*(m'), c \rangle = \langle m', \bar \phi_\kk(c) \rangle$ and hence $\bar\phi_\kk(c) = a^2 c'$. The other axioms of a Hopf algebra morphism are routine to check, and yield no additional requirements.    
\end{proof}

\subsection{The moduli space}

One can verify that isomorphism classes of $(n|1)$-dimensional supertori for $c \neq 0$ depend only on the line through $c \in N_\kk$, up to automorphisms $\GL(T_N)$ of the torus $T_N$. That is, the moduli space of $(n|1)$-dimensional supertori is the disjoint union of $\PP(N_\kk) / \GL(T_N) \cong \PP^{n-1}/\GL(n,\ZZ)$ and a point (representing the $c=0$ case).

\subsection{Supertori with higher odd dimension}

If $T = (T_N, \t)$ is a supertorus with $\t = \Span_\kk ( x_1, ..., x_n ) \oplus \Span_\kk ( \theta_1, ..., \theta_s )$, then its supergroup structure is determined entirely by the parameters $c_{ij} \in N_\kk$ representing the coefficients of $\frac{1}{2}[\theta_i, \theta_j] = (c_{ij})_1 x_1 + ... + (c_{ij})_n x_n$. Since $[\cdot, \cdot]$ is symmetric on $\t_1$, the $c_{ij}$ arrange themselves into a symmetric $s \times s$ matrix (with entries in $N_\kk$).

\begin{definition}
    Let $N$ be a lattice, and $(c_{ij})$ a symmetric square matrix with entries in $N_\kk$. Denote by $T_{N, (c_{ij})}$ the supertorus with the given parameters.
\end{definition}

The supertorus $T_{N,(c_{ij})}$ has coordinate Hopf superalgeba $\kk[M, \xi_1, ..., \xi_s]$ where $M = N^*$ and $s$ is the order of the matrix $(c_{ij})$. The Hopf algebra structure is given by
\begin{align*}
    \Delta x^m &= x^m \otimes x^m \left(1 + \sum_{i,j=1}^s \langle m, c_{ij} \rangle \xi_i \otimes \xi_j \right) & S x^m &= x^{-m} & \epsilon x^m &= 1 \\
    \Delta \xi_i &= \xi_i \otimes 1 + 1 \otimes \xi_i &
    S \xi_i &= -\xi_i & \epsilon \xi_i &= 0
\end{align*}
and likewise the group law on $A$-points is given by
\begin{align*}
    &(t_1, ..., t_n \mid \xi_1, ..., \xi_s) \cdot (u_1, ..., u_n \mid \eta_1, ..., \eta_s) \\
    &= \left(t_1 u_1 \left(1+ \sum_{i,j=1}^s (c_{ij})_1 \xi_i \eta_j \right), ..., t_n u_n \left(1+ \sum_{i,j=1}^s (c_{ij})_n \xi_i \eta_j\right) \; \Bigg| \; \xi_1 + \eta_1, ..., \xi_s + \eta_s \right)
\end{align*}
for $t_i, u_i \in (A_0)^\times$ and $\xi_, \eta_i \in A_1$.

The following lemma generalizes Lemma \ref{MorphismOfSupertoriLemma}, and can be proven in a virtually identical way.

\begin{lemma}
    Let $\phi : T_{N, (c_{ij})} \to T_{N', (c_{ij}')}$ be a morphism of supervarieties. Then $\phi$ is a morphism of supergroups if and only if all of the following conditions hold.
    \begin{enumerate}[label=(\roman*)]
        \item There is a lattice map $\bar \phi : N \to N'$ such that $\phi|_{T_N} = \bar \phi \otimes 1_{\kk^\times}$.
        \item The dual map $\phi^* : \kk[M', \xi_1', ..., \xi_{s'}'] \to \kk[M, \xi_1, ..., \xi_s]$ satisfies $\phi^*(\xi'_i) = \sum_{j=1}^s a_{ij} \xi_j$ for some $a_{ij} \in \kk$.
        \item $\bar \phi_\kk = \bar \phi \otimes 1_\kk : N_\kk \to N_\kk'$ satisfies $\bar \phi(c_{k\ell}) = \sum_{i,j=1}^{s'} a_{ik} a_{j\ell} c_{ij}'$ for all $k,\ell = 1, ..., s$.
    \end{enumerate}
\end{lemma}

If we define a \textbf{supertorus datum} as a tuple $(N, (c_{ij})_{i,j=1, ..., s})$, then we can formulate a category of supertorus data with the above objects, and morphisms $(\phi, (a_{ij})_{i=1, ..., s';\, j=1, ..., s}) : (N, (c_{ij})) \to (N', (c_{ij}'))$ as in the above lemma. Then the lemma yields the following result:
\begin{proposition}
    The category of algebraic supertori is equivalent to the category of supertorus data.
\end{proposition}

\section{Affine toric supervarieties}\label{ATSVSection}

\begin{definition}\label{TSVDef}
    A \textbf{toric supervariety} is a supervariety $X$ containing a supertorus $T$ as an open subscheme such that the action of $T$ on itself extends to an action of $T$ on the whole supervariety.
\end{definition}

\begin{lemma}
    Let $X$ be a toric supervariety with torus $T$. Then the underlying variety $X_0$ is a toric variety with torus $T_0$.
\end{lemma}

\begin{proof}
    The action of $T_0$ on $X$ preserves the defining ideal $\I$ of $X_0$.
\end{proof}

In this paper, we will assume $X$ is $(n|1)$-dimensional, so by Proposition \ref{IntegralIfOneOddDim} we may use $X_0$ and $\overline{X}$ interchangeably. We will begin by assuming $X$ is affine, which by Proposition \ref{Affine} is equivalent to the assumption that $X_0$ is affine.

\subsection{From a finite set in a monoid}

Given a supertorus $T'=T_{N',c'}$, let $M' = (N')^*$ be the character lattice of $T_0'=T_{N'}$ and $\A = \{m_1, ..., m_r\}$ a finite set in $M'$. We wish to construct a super-analog of the variety $Y_{\A}$ described in section \ref{ATVBackground}. Define
\begin{align*}
    \A' = \B' &= \{m_i \in \A \mid \langle m_i, c' \rangle \neq 0\} = \{m_1, ..., m_q\}  \\ \A'' &= \{m_i \in \A \mid \langle m_i, c' \rangle = 0\} = \{m_{q+1}, ..., m_r\}
\end{align*}
so that $\A = \A' \sqcup \A''$. Moreover, let $\B'' = \{n_{q+1}, ..., n_s\}$ be a finite set in the monoid $\NN \A''$, and set $\B = \B' \sqcup \B''$. (We assume henceforth that $\B$ is nonempty so that we do not degenerate to the classical, purely-even case.) We will write $n_j = m_j$ for $j=1, ..., q$ so that $\B = \{n_1, ..., n_s\}$, and $n_{j} = \sum_{i=q+1}^r p_i^j m_i$ for $j=q+1, ..., s$ for some $p_{i}^j \in \NN$. We can do the same for $j=1, ..., q$, in which case we have $p_i^j = 1$ if $i=j$ and $p_i^j=0$ otherwise.

Now define a morphism of supervarieties
$$\Phi_{\A,\B} : T' \to \kk^{r|s}$$
at the level of $A$-points (for $A$ a superalgebra) by
$$\Phi_{\A,\B}(t \mid \xi) = (t^{m_1}, ..., t^{m_r} \;|\; t^{n_1}\xi, ..., t^{n_s}\xi) \in \kk^{r|s}(A)$$
for $t \in ((A_0)^\times)^{\dim N'}$ and $\xi \in A_1$. Denote by $J \subseteq \kk[x_1, ..., x_r, \xi_1, ..., \xi_s]$ the ideal of functions vanishing on the image of $\Phi_{\A,\B}$. Then the Zariski closure of the image is defined as the closed subvariety $V(J) := \Spec \kk[x_1, ..., x_r, \xi_1, ..., \xi_s] / J$ of $\kk^{r|s}$, henceforth referred to as $Y_{\A,\B}$.

\begin{proposition}\label{YA is a TSV Prop}
    Given $\A, \B \subseteq M$ as above, $Y_{\A,\B}$ is an affine toric supervariety whose supertorus is $\Phi_{\A,\B}(T') = T_{(\ZZ \A)^*, \bar \Phi_{\A,\B}(c')}$, where $\bar \Phi_{\A,\B}(c')$ is defined in as in Lemma \ref{MorphismOfSupertoriLemma}.
\end{proposition}

\begin{proof}
    Let $T'' = (\kk^\times)^r \times \kk^{0|s} \subset \kk^{r|s}$ as a supervariety without any assumption of a supergroup structure. We can view $\Phi_{\A,\B}$ as a map of supergroups $T' \to T''$ in the following way: We may freely change the underlying classical variety of the codomain to $(\kk^\times)^r$. Now it remains to impose a supertorus structure on $T''$, so let $A$ be a superalgebra and let $t,u \in ((A_0)^\times)^{\dim N'}$ and $\xi,\eta \in A_1$. If $\Phi_{\A,\B}$ is to be a group homomorphism, we must have
    \begin{align*}
        &(t^{m_1}, ..., t^{m_r} \mid t^{n_1} \xi, ..., t^{n_s} \xi) \cdot (u^{m_1}, ..., u^{m_r} \mid u^{n_1} \eta, ..., u^{n_s} \eta) \\
        &= \Phi_{\A,\B}(t \mid \xi) \cdot \Phi_{\A,\B}(u \mid \eta) \\
        &= \Phi_{\A,\B}((t \mid \xi) \cdot (u \mid \eta)) \\
        &= \Phi_{\A,\B}(t_1u_1(1+c_1' \xi\eta), ..., t_n u_n(1+c_n' \xi \eta) \mid \xi+\eta) \\
        &= \left(t^{m_1}u^{m_1} (1+\langle m_1, c' \rangle \xi \eta), ..., t^{m_r}u^{m_r} (1+\langle m_r, c' \rangle \xi \eta) \mid t^{n_1}u^{n_1} (\xi+\eta), ..., t^{n_s}u^{n_s} (\xi+\eta) \right)
    \end{align*}
    where we identify $T_0'$ with $(\kk^\times)^n$ and $M'$ with $\ZZ^n$. Thus, a potential group law on $T''(A)$ is $$(a_1, ..., a_r \mid \alpha_1, ..., \alpha_s) \cdot (b_1, ..., b_r \mid \beta_1, ..., \beta_s) = (d_1, ..., d_r \mid \delta_1, ..., \delta_s)$$ where
    \begin{align*}
        d_i &= a_i b_i + \langle m_i, c' \rangle \alpha_i \beta_i &&\text{for } i=1, ..., q \\
        d_i &= a_i b_i &&\text{for } i=q+1, ..., r \\
        \delta_j &= b^{p^j} \alpha_j + a^{p^j} \beta_j &&\text{for } j=1, ..., s.
    \end{align*}
    For $j=1, ..., q$, the latter simplifies to
    \begin{align*}
        \delta_j &= b_j \alpha_j + a_j \beta_j.
    \end{align*}
    It is routine to check that this operation makes $T''(A)$ into a group with identity element $(1, ..., 1 \mid 0, ..., 0)$, and moreover $T''$ into an $(r|s)$-dimensional supertorus  which decomposes as $Q(1)^q \times T'''$ for some $(r-q | s-q)$-dimensional supertorus $T'''$.
    
    It follows by Lemma \ref{ImageLem} that $\Phi_{\A,\B}(T')$ is a supertorus closed in $T''$, so $Y_{\A,\B} \cap T'' = \Phi_{\A,\B}(T')$ because $Y_{\A,\B}$ is the closure of $\Phi_{\A,\B}(T')$. As in the classical setting, $\Phi_{\A,\B}(T')$ is open in $Y_{\A,\B}$ and irreducible, so $Y_{\A,\B}$ is also irreducible.

    We now verify that $\Phi_{\A,\B}(T')$ acts on $Y_{\A,\B}$ via the functor of points. Since $\Phi_{\A,\B}(T') \subset T''$, the action is by homeomorphisms and so takes supervarieties to supervarieties. For a superalgebra $A$, it follows that $\Phi_{\A,\B}(T')(A) = (t, \xi) \cdot \Phi_{\A,\B}(T')(A) \subseteq (t, \xi) \cdot Y_{\A,\B}(A)$ for $t \in (A_0^*)^{\dim N'}, \xi \in A_1$. Hence $(t,\xi) \cdot Y_{\A,\B}(A)$ is the functor of points of a supervariety containing $\Phi_{\A,\B}(T')$, meaning $Y_{\A,\B}(A) \subseteq (t,\xi) \cdot Y_{\A,\B}(A)$ by definition of the Zariski closure. Replacing $(t,\xi)$ with its inverse $(t^{-1}, -\xi)$ yields the reverse inclusion, so in fact $Y_{\A,\B}(A) = (t,\xi)Y_{\A,\B}(A)$. Therefore $Y_{\A,\B}$ is an affine toric supervariety with torus $\Phi_{\A,\B}(T')$.

    As in the usual setting, the even part $\Phi_{\A,\B}(T')_0$ of the torus will have character lattice $\ZZ \A$. Since we assume $s>0$, we also have one dimension of odd functions (generated by $\xi$). Viewing $\Phi_{\A,\B}$ as a morphism $T' \to \Phi_{\A,\B}(T')$, we see that $\Phi_{\A,\B}^*(\xi) = \xi$ and hence the additional parameter $c$ is $\bar \Phi_{\A,\B}(c')$ by Lemma \ref{MorphismOfSupertoriLemma}.
\end{proof}

\subsection{The ideal of an affine toric supervariety}

As in the usual case, we are interested in describing the ideal $J(Y_{\A,\B})$ of functions that vanish on $Y_{\A,\B}$. The map $\Phi_{\A,\B}$ can be viewed as a map to an $(r|s)$-dimensional torus, inducing a morphism of lattices
$$\hat \Phi_{\A,\B} : \ZZ^{r|s} \to M' \oplus \ZZ^{0|1}$$
mapping the coordinate vectors of the domain as
\begin{align*}
    e_i &\mapsto (m_i \mid 0) &&i=1, ..., r \\
    e_j' &\mapsto (n_j \mid 1) &&j=1, ..., s.
\end{align*}
The kernel $L$ of this morphism records the relations among the $m_i$ and $n_j$, so if $(\ell \mid \ell') \in L$, then $\sum_{i=1}^r \ell_i (m_i \mid 0) + \sum_{j=1}^s \ell_j' (n_j \mid 1) = 0$. Notice that $\sum_{j=1}^s \ell_j' = 0$ due to the ``odd" coordinate.

For $\ell = (\ell_1, ..., \ell_r \mid \ell_1', ..., \ell_s') \in L$, set
\begin{align*}
    \ell_+ &= \sum_{\ell_i>0} \ell_i e_i &\ell_- &= -\sum_{\ell_i<0} \ell_i e_i &\ell_+' &= \sum_{\ell_j'>0} \ell_j' e_j' &\ell_-' &= -\sum_{\ell_j'<0} \ell_j' e_j'
\end{align*}
so that $\ell = (\ell_+ - \ell_-) + (\ell_+' - \ell_-')$ and $\ell_\pm \in \NN^r, \ell_\pm' \in \NN^s$. It follows immediately that the binomial
\begin{align*}
    x^{\ell_+} \xi^{\ell_+'} - x^{\ell_-} \xi^{\ell_-'}
\end{align*}
vanishes on $\Phi_{\A,\B}(T')$ and hence also on its closure $Y_{\A,\B}$. Since each $\xi_i$ is sent to a multiple of $\xi$, the monomials $\xi_i \xi_j$ also vanish on $Y_{\A,\B}$. Thus, it is sufficient to consider only those binomials which contain either no $\xi_j$ or exactly one in each term.

\begin{proposition}\label{IdealProposition}
    The ideal of the affine toric supervariety $Y_{\A,\B} \subseteq \kk^{r|s}$ is
    \begin{align*}
        J(Y_{\A,\B}) &= (x^{\ell_+} \xi^{\ell_+'} - x^{\ell_-} \xi^{\ell_-'} \mid \ell \in L) + (\xi_i \xi_j) \\
        &= (x^{k_+} \xi^{k_+'} - x^{k_-} \xi^{k_-'} \mid k_\pm \in \NN^r, k_\pm' \in \NN^s \text{ and } k_+ - k_- + k_+' - k_-' \in L) + (\xi_i \xi_j)
    \end{align*}
\end{proposition}

\begin{proof}
    The proof is largely the same as in the classical case: Firstly, the two ideals on the right are equal by an identical proof, and we let $J_L \subseteq J(Y_{\A,\B})$ denote this ideal. Fixing identifications $T_0' \cong (\kk^\times)^n$ and $M' \cong \ZZ^n$, the map $\Phi_{\A,\B} : T' \to \kk^{r|s}$ is given by monomials $t^{m_i}$ for $i=1, ..., r$ and $t^{n_j} \xi$ for $j=1, ..., s$ in the variables $t_1, ..., t_n, \xi$.
    
    Now pick a monomial order $\geq$ on $\kk[x_1, ..., x_r, \xi_1, ..., \xi_s]$, say lexicographic order with $x_1>...>x_r>\xi_1>...>\xi_s$. While this is not precisely a monomial order in the traditional sense due to the presence of odd variables, we will not invoke any properties that do not hold in this situation. If $J_L \neq J(Y_{\A,\B})$, we can pick $f \in J(Y_{\A,\B}) \backslash J_L$ with minimal leading monomial $x^a = \prod_{i=1}^r x_i^{a_i}$ or $x^a \xi_j = \left( \prod_{i=1}^r x_i^{a_i} \right) \xi_j$. We may assume without loss of generality that $f$ has leading coefficient 1, and that $f$ contains no summands in $(\xi_i \xi_j)$.

    Since $f(t^{m_1}, ..., t^{m_r}, t^{n_1}\xi, ..., t^{n_s}\xi) \equiv 0$ as a polynomial in $t_1, ..., t_n, \xi$ and $f$ contains no summands in $(\xi_i \xi_j)$, there must be cancellation involving the leading term. Hence $f(x_1, ..., x_r, \xi_1, ..., \xi_s)$ contains a monomial $x^b < x^a$ or $x^b \xi_k < x^a \xi_j$, respectively, such that
    \begin{align*}
        \prod_{i=1}^r (t^{m_i})^{a_i} = \prod_{i=1}^r (t^{m_i})^{b_i} \quad\quad\text{or} \quad\quad \left( \prod_{i=1}^r (t^{m_i})^{a_i}\right) t^{n_j} \xi= \left(\prod_{i=1}^r (t^{m_i})^{b_i}\right) t^{n_k} \xi,
    \end{align*}
    implying
    \begin{align*}
        \sum_{i=1}^r a_i m_i = \sum_{i=1}^r b_i m_i \quad\quad\text{or} \quad\quad \sum_{i=1}^r a_i m_i + n_j = \sum_{i=1}^r b_i m_i + n_k.
    \end{align*}
    Therefore $a-b \in L$ or $(a+e_j') - (b+e_k') \in L$, so by the second description of $J_L$ it follows that $f - x^a + x^b$ or $f-x^a\xi_j + x^b\xi_k$ also lies in $J(Y_{\A,\B}) \backslash J_L$ but has strictly smaller leading term. With this contradiction, we are finished.
\end{proof}

\begin{remark}
In the classical setting, there is a nice classification of affine toric varieties via toric ideals (i.e.\ prime ideals generated by binomials arising from a lattice as above). A super-analog of this result is possible but rather less interesting, since it fails to describe the torus action on the corresponding supervariety, and in particular the parameter $c \in N_\kk$. Additionally, our constraint of one odd dimension makes the description quite cumbersome.
\end{remark}

\subsection{From an affine semigroup}

Recall from section \ref{ATVBackground} that an affine toric variety is isomorphic to $\Spec \kk[S]$ for some affine semigroup $S$, and every $\Spec \kk[S]$ has the structure of an affine toric variety. The generalization to the super case with one odd dimension requires the same data with one minor addition.

Let $T_N$ be the ordinary torus with character lattice $N^* = M = \ZZ S$. For $c \in N_\kk$, let $J_c$ be the $T_N$-invariant ideal $(S - \ker c) = (x^m \in \kk[S] \mid \langle m, c \rangle \neq 0)$ of $\kk[S]$, where $S-\ker c$ denotes the set difference. Note that $J_c \neq 0$ unless $c=0$.

\begin{lemma}\label{UniqueGenSetLem(b)}
\begin{enumerate}
    \item Let $S$ be an affine semigroup, and let $J \subseteq \kk[S]$ be a nonzero $T_N$-invariant ideal that contains $J_c$. Then $\Spec (\kk[S] \oplus \xi J)$ is an affine toric supervariety whose torus is $T_{N,c}$.

    \item The ideal $J$ has a finite generating set consisting of monic monomials. Moreover, there is a minimal generating set of this type, unique up to multiplication by invertible monomials in $\kk[S]$. In particular, if $S$ is a pointed semigroup (i.e.\ the only unit is 1), then there is a unique minimal generating set.
    
    \item Suppose $S = \NN \A$ for a finite set $\A$ in a lattice $M'$ ($\supset M = \ZZ S$) such that $\A = \A' \sqcup \A''$ as above where $\A'' = \{m \in \A \mid \langle m,c \rangle_{M \times N_\kk} = 0\}$. Let $\B''$ be a finite set in $\NN \A''$ such that $J = (x^{n_j} \mid n_j \in \A' \sqcup \B'')$. Then $\Spec (\kk[S] \oplus \xi J) \cong Y_{\A,\B}$.
\end{enumerate}
\end{lemma}

\begin{proof}
\begin{enumerate}
    \item Follows from (c) and Proposition \ref{YA is a TSV Prop}.
    
    \item Since $J$ is $T_N$-invariant, it has a generating set by monic monomials. For finiteness, notice that $\kk[S]$ is a quotient of some $\kk[y_1, ..., y_k]$, which is Noetherian. It therefore remains to prove uniqueness.
    
    \tab Let $\S_1$ and $\S_2$ be two generating sets by monic monomials, and let $x^{m_1} \in \S_1 \backslash \S_2$. Since $x^{m_1} \in J$, we must have $x^{m_1} = a x^{m_2}$ for some monomial $a \in \kk[S]$ and $m_2 \in \S_2$. Since $x^{m_2} \in J$, we must have $x^{m_2} = b x^{m_3}$ for some monomial $b \in \kk[S]$ and $m_3 \in \S_1$. Hence $x^{m_1} = ab x^{m_3}$ and so either $m_1 = m_3$ or $x^{m_1}$ is redundant in $S_1$. In the event that $m_1 = m_3$, it holds that $a=b^{-1}$, so $x^{m_1} \in \S_1$ and $x^{m_2} \in \S_2$ differ only by an invertible monomial.

    \item We first observe that the defining condition $\langle m,c \rangle_{M \times N_\kk} = 0$ of $\A''$ is equivalent to $\langle m, c' \rangle_{M' \times N_\kk'} = 0$ because $M = \ZZ \A$ is a sublattice of $M'$ via $\bar \phi^* : M \to M'$ and likewise $N$ is a quotient of $N'$ via $\bar \phi : N' \to N$.
    
    \tab The morphism $\Phi_{\A,\B} : T_{(M')^*, c'} \to \kk^{r|s}$ corresponds to a map
    $$\pi : \kk[x_1, ..., x_r, \xi_1, ..., \xi_s] \to \kk[M', \xi]$$
    given by $x_i \mapsto x^{m_i}$ and $\xi_j \mapsto x^{n_j} \xi$. The image of $\pi$ is exactly $\kk[S] \oplus \xi J$ because $J$ is the ideal generated by the $x^{n_j}$. Hence, if we can show that $\ker \pi$ is the ideal $J(Y_{\A,\B})$, then
    \begin{align*}
        \kk[S] \oplus \xi J &= \Im \pi \\
        &\cong \kk[x_1, ..., x_r, \xi_1, ..., \xi_s]/ \ker \pi \\
        &= \kk[x_1, ..., x_r, \xi_1, ..., \xi_s]/ J(Y_{\A,\B}) \\
        &= \kk[Y_{\A,\B}]
    \end{align*}
    and so (c) will hold. But this follows by Proposition \ref{IdealProposition}. \qedhere
\end{enumerate}
\end{proof}

The condition that $J$ is nonzero is required only in the case $c=0$, in order to ensure that the resulting supervariety is not purely even.

\begin{example}
    If $S = \NN$ and $c = 1$, then we obtain the supertorus $T_{\ZZ,1} \cong Q(1)$ and the ordinary toric variety $\Spec \kk[\NN] \cong \Spec \kk[x] \cong \AA^1$. We observe that $J_c = (x)$, so $\Spec \kk[x, x\xi]$ and $\Spec \kk[x, \xi]$ are two toric supervarieties with torus $Q(1)$. While these two spaces are isomorphic as supervarieties, they admit different torus actions and are in fact not isomorphic as toric supervarieties. For more on morphisms of toric supervarieties, see Section \ref{MorphismsSection}.
\end{example}

The data of the ideal $J$ is equivalent to the data of its minimal generating set by monomials, which can be characterized as follows.

\begin{definition}\label{AdmissibleDef}
    Let $S$ be an affine semigroup and $T_{(\ZZ S)^*,c}$ a supertorus. For $m,m' \in S$, we write $m' \leq_S m$ if there exists $m'' \in S$ such that $m'+m''=m$. A nonempty finite set $\B \subset S$ is \textbf{$c$-admissible} if:
    \begin{enumerate}[label=(\roman*)]
        \item For all $m \in S - \ker c$, there is $b \in \B$ such that $b \leq_S m$.
    \end{enumerate}
    It is moreover \textbf{minimal} if:
    \begin{enumerate}[label=(\roman*)]
        \setcounter{enumi}{1}
        \item If $a,b \in \B$ satisfy $a \leq_S b$, then $a=b$.
    \end{enumerate}
\end{definition}

\begin{lemma}
    The nonempty finite set $\B \subset S$ is (minimal) $c$-admissible if and only if it is a (minimal) generating set for a nonzero ideal $J \subseteq \kk[S]$ containing $J_c$.
\end{lemma}
\begin{proof}
    Condition (1) is equivalent to the ideal generated containing $J_c$, and condition (2) is equivalent to minimality.
\end{proof}

\begin{proposition}\label{TFAEProp}
    Let $X$ be an affine supervariety. Then the following are equivalent:
    \begin{enumerate}
        \item $X$ is an affine toric supervariety whose torus is $T = T_{N,c}$.
        
        \item $X = Y_{\A,\B}$ for a finite set $\A = \A' \sqcup \A''$ in a lattice and $\B = \A' \sqcup \B''$ nonempty where $\B''$ is a finite set in $\NN \A''$ and $\A'' = \{m \in \A \mid \langle m, c \rangle = 0\}$.
        
        \item $X = \Spec (\kk[S] \oplus \xi J)$ for an affine semigroup $S$ and a nonzero $T_N$-invariant ideal $J \subseteq \kk[S]$ that contains $J_c$.
        
        \item $X = \Spec \kk[S, \xi \B]$ for an affine semigroup $S$ and an $c$-admissible set $\B \subset S$.
    \end{enumerate}
\end{proposition}
\begin{proof}
    We have already seen that (d) $\iff$ (c) $\iff$ (b) $\implies$ (a). Now suppose (a) holds, so $X$ is an affine toric supervariety containing the supertorus $T$. The inclusion $T \hookrightarrow X$ induces a map of coordinate rings $$\kk[X] \to \kk[T]$$
    which is injective by part (iv) of Definition \ref{SupervarietyDefinition} since $T$ is an open subscheme of the supervariety $X$. Thus we may view $\kk[X]$ as a subalgebra of $\kk[T]$ and in particular $\kk[X_0]$ as a subalgebra of $\kk[T_0]$. From our knowledge of the usual case, this implies $\kk[X_0] = \kk[X]_0 = \kk[S]$ for an affine semigroup $S$.

    Since $T$ acts on $X$, it also acts on $\kk[X]$, so that $\kk[X]$ is a subrepresentation of $\kk[T]$. For each monomial $x^m \xi \in \kk[X]$, it holds that $x^m \in \kk[X]$ as well, so $\kk[X]_1 = \xi J$ for some $T_N$-invariant ideal $J \subset \kk[S]$. If $L(m)$ is irreducible (i.e.\ if $\langle m, c \rangle \neq 0$), then $x^m \in \kk[X]$ implies $x^m \xi \in \kk[X]$, meaning $J$ must contain $J_c$. Therefore (c) follows from (a).
\end{proof}

\begin{remark}
    The $c$-admissible set $\B$ can be taken as the finite set $\B$ occurring in $Y_{\A,\B}$, but it is not in general minimal: Consider the supertorus $Q(1) \cong T_{\ZZ, 1} \cong \Spec \kk[x^{\pm 1},\xi]$ as a toric supervariety containing itself. Here, we may take $S = \ZZ$, so that $X \cong Y_{\A,\B}$ for $\A = \B = \{\pm 1\}$ because $\B$ must contain $\A = \A - \ker c$. However, such $\B$ cannot be minimal, since $1$ and $-1$ generate the same ideal of $\kk[x^{\pm 1}]$.
\end{remark}

\begin{example}
    For every affine toric variety $X_0 = \Spec \kk[S] \supset T_N$ and supertorus $T_{N,c} \supset T_N$, there exists a \textbf{trivial} toric supervariety structure given by $J = \kk[S]$.
\end{example}

\begin{example}
    Suppose $T=T_{N,c}$ is an indecomposable supertorus. Then for a given affine toric variety $X_0 = \Spec \kk[S]$ with torus $T_N$, $J_c$ is generated by all non-constant monomials in $\kk[S]$. If $S$ is not pointed (i.e.\ it contains two distinct elements which are inverses), then $J_c = (1)$ and the only toric supervariety whose underlying variety is $X_0$ and whose torus is $T$, is the trivial one.

    If $S$ is pointed, then $J_c$ is a maximal ideal, so there is a unique nontrivial toric supervariety structure, in which $J=J_c$.
\end{example}

\subsection{Some further computations}
We have seen from Proposition \ref{TFAEProp} that an affine toric supervariety $X = \Spec A \supset T_{N,c}$ with one odd dimension consists of the data of an ordinary affine toric variety $X_0 = \Spec A_0 \supset T_N$ together with the additional data of a nonzero $T_N$-invariant ideal $J \subset A_0$ that contains $J_c$. In this case, it holds that $A = A_0 \oplus A_1$ where $A_1 = \xi J$. In this section, we provide some further information about the \textbf{fermionic ideal} $J$.

\begin{corollary}\label{ClassificationCor}
Let $X_0 = \Spec A_0$ be an affine toric variety with torus $T_0=T_N$. Affine toric supervarieties with underlying variety $X_0$ correspond to nonzero $A_0 \times T_0$-bimodules $J$ such that $J_c \subseteq J \subseteq A_0$. Equivalently if $c \neq 0$, they correspond to $T_N$-invariant submodules (or quotients) of $A_0/J_c$.
\end{corollary}

\begin{example}
    Let $T = T_{\ZZ^2, (1,0)} \cong Q(1) \times \kk^\times$ and $X_0 = \Spec A_0$ for $A_0 = \kk[x_1^{\pm 1} x_2, x_2]$. Here, we have $J_{(1,0)} = (x_1^{\pm 1} x_2)$, so $A_0/J_{(1,0)} = \kk[x_2]/(x_2^2)$. This module has three torus-invariant ideals, namely $(0), (x_2)$, and itself. These ideals yield toric supervarieties whose coordinate rings are, respectively, $\kk[x_1^{\pm 1} x_2, x_2, x_1^{\pm 1} x_2 \xi]$, $\kk[x_1^{\pm 1} x_2, x_2, x_1^{\pm 1} x_2 \xi, x_2 \xi]$, and $\kk[x_1^{\pm 1} x_2, x_2, \xi]$.

    Let us consider the toric supervariety $X=\Spec A$ for $A= \kk[x_1^{\pm 1} x_2, x_2, x_1^{\pm 1} x_2 \xi]$ in more detail. To witness it as a subvariety of some $\kk^{r|s}$, we take the three generators of $A_0$ and the two generators of $A_1$, arranging them into finite sets as follows:
    \begin{align*}
        \A' = \B' &= \{(1,1), (-1,1)\} \\
        \A'' &= \{(0,1)\} \\
        \B'' &= \varnothing
    \end{align*}
    Then we obtain
    \begin{align*}
        \Phi_{\A,\B} : T_{\ZZ^2, (1,0)} &\to \kk^{3|2} \\
        \Phi_{\A,\B}(t_1, t_2, \xi) &= (t_1t_2, t_1^{-1}t_2, t_2 \mid t_1t_2 \xi, t_1^{-1} t_2 \xi)
    \end{align*}
    and indeed $X \cong Y_{\A,\B}$.
\end{example}

To illustrate the naturality of the data of the fermionic ideal, we may verify that the Duflo-Serganova functor $$DS_Q : \kk[X] \mapsto \frac{\ker Q|_{\kk[X]^{Q^2}}}{Q(\kk[X]^{Q^2})}$$ of \cite{DS} outputs the data of $J$ when $Q= \xi \left(c_1 x_1 \d{x_1} + ... + c_n x_n \d{x_n} \right) + \d\xi$ is the odd vector field arising from the action of $\t_1$ on $\kk[X]$.

\begin{proposition}\label{DSProp}
    With notation as above, $DS_Q \kk[X] \cong A_0/J$.
\end{proposition}

\begin{proof}
    We may assume that $\kk[X] \subseteq \kk[x_1^{\pm 1}, ..., x_n^{\pm 1}, \xi]$. Since $Q^2 = c_1 x_1 \d{x_1} + ... + c_n x_n \d{x_n}$, it holds that $Q^2(f(x)+g(x) \xi) = Q^2(f(x)) + Q^2(g(x))\xi$. For a monomial $x^m$, we obtain $Q^2(x^m) = \langle m, c \rangle x^m$ and hence
    $$\kk[X]^{Q^2} = (H \cap \kk[X]_0) \oplus (\xi H \cap \kk[X]_1)$$
    where $$H = \bigoplus_{m \in \ker c} \kk x^m$$
    is the ``hyperplane" orthogonal to $c$.

    We observe that $Q(H) = 0$ and $Q(g(x) \xi) = g(x)$, so
    \begin{align*}
        Q(\kk[X]^{Q^2}) &= Q(H \cap \kk[X]_0) + Q(\xi H \cap \xi J) \\
        &= H \cap J
    \end{align*}
    and $\ker Q|_{\kk[X]^{Q^2}} = H \cap \kk[X]_0$. Hence
    \begin{align*}
        DS_Q \kk[X] &= \frac{H \cap \kk[X]_0}{H \cap J} \\
        &\cong \kk[X]_0 / J,
    \end{align*}
    where the final isomorphism is due to the fact that $J \supseteq J_c$ and $H + J_c = \kk[X]_0$.  
\end{proof}

\subsection{Morphisms}\label{MorphismsSection}
Let $X = \Spec (\kk[S] \oplus \xi J)$ and $X' = \Spec (\kk[S'] \oplus \xi' J')$ be affine toric supervarieties with supertori $T = T_{N,c}$ and $T'=T_{N',c'}$, respectively. Recall that a morphism of supergroups $T \to T'$ consists of a morphism $\bar\phi : N \to N'$ of lattices and a number $a \in \kk$ such that $\bar \phi_\kk(c) = a^2 c'$.
\begin{definition}\label{MorphismDef}
    A morphism $\phi : X \to X'$ of supervarieties is called \textbf{toric} if $\phi(T) \subseteq T'$ and $\phi|_T : T \to T'$ is a morphism of supergroups.
\end{definition}

Before we characterize toric morphisms, we will establish some notation. The corresponding maps on coordinate superalgebras are
\begin{align*}
    \phi^* &: \kk[S'] \oplus \xi' J' \to \kk[S] \oplus \xi J \\
    \phi|_T^* &: \kk[M', \xi'] \to \kk[M, \xi]
\end{align*}
where $M = \ZZ S$ and $M' = \ZZ S'$. Since the conditions in the above definition imply that $\phi_0 : X_0 \to X_0'$ is a toric morphism, it follows that $\phi_0^* : \kk[S'] \to \kk[S]$ arises from a morphism $\hat \phi^* : S' \to S$ of semigroups, which extends to a morphism $\bar \phi^* : M' \to M$ of lattices. We write $\bar \phi : N \to N'$ for the dual map of $\bar \phi^*$.

\begin{proposition}\label{MorphismProp}
Let $T$ and $T'$ be the supertori of the affine toric supervarieties $X$ and $X'$ as above.
\begin{enumerate}
    \item A morphism $X \to X'$ of supervarieties is toric if and only if $\phi^*_0$ arises from a semigroup morphism $\hat \phi^* : S' \to S$, and $\bar \phi_\kk(c) = a^2 c'$ where $a \in \kk$ is the $\xi$-coefficient of $\phi^*(\xi')$.

    \item A morphism $\phi = (\bar\phi,a) : T \to T'$ of supergroups extends to a morphism $\Spec \kk[S] \to \Spec \kk[S']$ of toric supervarieties if and only if $\bar \phi^*(S') \subseteq S$, and either $a=0$ or $\phi^*(J') \subseteq J$.

    \item A toric morphism is equivariant, meaning for any superalgebra $A$, it holds that $\phi(t \cdot x) = \phi(t) \cdot \phi(x)$ for all $t \in T(A)$ and $x \in X(A)$.
\end{enumerate}
\end{proposition}
\begin{proof}
    The proof is the same as in the usual case (\cite{CLS}, Proposition 1.3.14), except that we use Lemma \ref{MorphismOfSupertoriLemma} to determine when $\phi|_T$ is a morphism of supergroups. Notice that the condition $a=0$ or $\phi^*(J') \subseteq J$ is equivalent to the condition $\phi^*(\xi' J') \subseteq \xi J$, since $\phi^*(\xi') = a\xi$.
\end{proof}

\section{Quasinormal toric supervarieties}\label{QuasinormalSection}

We call a supervariety \textbf{quasinormal} if its underlying variety is normal, or equivalently if the local ring at every point is an integral superdomain $A$ such that $A/(A_1)$ is integrally closed.

\subsection{The affine case}

Suppose henceforth that $X_\sigma \supset T_N$ is the normal affine toric variety corresponding to the cone $\sigma \subset N_\RR$, so $X_\sigma = \Spec \kk[S_\sigma]$ where $S_\sigma$ is the semigroup of the dual cone $\check\sigma$. If $\tau$ is a face of $\sigma$, then $X_\sigma$ contains $X_\tau = \Spec \kk[S_\tau]$ as an open subvariety, and $\kk[S_\tau] = \kk[S_\sigma]_{x^{m_{\sigma, \tau}}}$ where $m_{\sigma, \tau} \in S_\sigma$ is such that $\tau = \{u \in \sigma \mid \langle m_{\sigma, \tau}, u \rangle = 0\}$.
\begin{proposition}\label{LocalizeProp}
    Using the above notation, suppose $X_{\sigma, J^\sigma} = \Spec (\kk[S_\sigma] \oplus \xi J^\sigma)$ is the affine toric supervariety with underlying variety $X_\sigma$ and fermionic ideal $J^\sigma$.
    \begin{enumerate}
        \item The localization $$(\kk[S_\sigma] \oplus \xi J^\sigma)_{x^{m_{\sigma, \tau}}} = \kk[S_\tau] \oplus \xi J^\tau$$ is the coordinate ring of the open subvariety $X_{\tau, J^\tau}$ obtained by restricting to the open affine $X_\tau$. In particular, it holds that $J^\sigma_{x^{m_{\sigma, \tau}}} = J^\tau$.
        \item Let $J^\sigma_c = (x^m \in \kk[S_\sigma] \mid \langle m, c \rangle \neq 0)$, and likewise for $\tau$. It holds that $(J^\sigma_c)_{x^{m_{\sigma, \tau}}} = J^\tau_c$.
        \item As defined in part (a), $X_{\tau, J^\tau}$ is an affine toric supervariety whose torus is $T_{N,c}$.
    \end{enumerate}
    
\end{proposition}
\begin{proof}
\begin{enumerate}
    \item The fact that $\kk[S_\sigma]_{x^{m_{\sigma, \tau}}} = \kk[S_\tau]$ is standard, so it remains to prove $J^\sigma_{x^{m_{\sigma, \tau}}} = J^\tau$. This follows immediately by computing sections over the basic open set defined by $x^{m_{\sigma, \tau}}$.

    \item Let $x^m$ be a generator of $J^\sigma_c$ and $k \in \NN$, so $\langle m, c \rangle \neq 0$ and $x^{m-km_{\sigma, \tau}} \in (J^\sigma_c)_{x^{m_{\sigma, \tau}}}$. Now consider the quantity $\langle m-km_{\sigma, \tau}, c \rangle.$ If $\langle m_{\sigma, \tau}, c \rangle = 0$, then indeed $x^{m-km_{\sigma, \tau}} \in J_c^\tau$. Otherwise $\langle m_{\sigma, \tau}, c \rangle \neq 0$, so $x^{\pm m_{\sigma, \tau}} \in J_c^\tau$ and hence $J_c^\tau=(1)$.

    \tab To see that $(J^\sigma_c)_{x^{m_{\sigma, \tau}}} \supseteq J^\tau_c$, let $x^m$ be a generator of $J_c^\tau$, so $x^m \in \kk[S_\tau]$ satisfies $\langle m, c \rangle \neq 0$. By part (a), $m=m'-km_{\sigma,\tau}$ for some $x^{m'} \in S_\sigma$ and $k \in \NN$. If $\langle m_{\sigma,\tau}, c \rangle = 0$, then $\langle m',c\rangle \neq 0$ and hence $x^{m'} \in J_c^\sigma$, so $x^m \in (J_c^\sigma)_{x^{m_{\sigma,\tau}}}$. Otherwise $\langle m_{\sigma, \tau}, c \rangle \neq 0$, so $x^{\pm m_{\sigma, \tau}} \in (J_c^\sigma)_{x^{m_{\sigma,\tau}}}$ and therefore $(J_c^\sigma)_{x^{m_{\sigma,\tau}}}=(1)$.

    \item This follows from part (b), as $(J^\sigma_c)_{x^{m_{\sigma, \tau}}} = J^\tau_c$. \qedhere
\end{enumerate}
\end{proof}

That is, in order to determine a quasinormal affine toric supervariety over the torus $T_{N,c}$, one requires the data of a strongly convex rational polyhedral cone $\sigma \subset N_\RR$ decorated by a nonzero $T_N$-invariant ideal $J^\sigma \supset J_c$ of $\kk[S_\sigma]$. When restricting to faces of $\sigma$, one finds that we may decorate each face with a certain localization $J^\tau$ of $J^\sigma$. These decorations will play an important role later, when we describe the non-affine case.

First, we translate Proposition \ref{MorphismProp} to the language of cones.

\begin{proposition}\label{ConeMorphismProp}
    Suppose we have strongly convex rational polyhedral cones $\sigma \subset N_\RR$ and $\sigma' \subset N'_\RR$, a morphism $\phi = (\bar\phi, a) : T_{N,c} \to T_{N',c'}$ of supergroups, and nonzero $T_N$-, $T_{N'}$-invariant ideals $J \supseteq J_c^\sigma$ and $J' \supseteq J_{c'}^{\sigma'}$ of $\kk[S_\sigma]$ and $\kk[S_{\sigma'}]$, respectively. Then $\phi$ extends to a map of affine toric supervarieties $\phi : X_{\sigma, J} \to X_{\sigma', J'}$ if and only if $\bar \phi_\RR(\sigma) \subseteq \sigma'$, and either $a=0$ or $\phi^*(J') \subseteq J$.
\end{proposition}

The following lemma will be useful in section \ref{FiberProductSection}, but is also interesting in its own right.

\begin{lemma}\label{JcPreservedLemma}
    Let $(\bar \phi, a) : X_{\sigma, J} \to X_{\sigma', J'}$ be a map of toric supervarieties, and let $$\phi^* : \kk[S_{\sigma'}] \oplus \xi' J' \to \kk[S_\sigma] \oplus \xi J$$ be the dual map. Then $\phi^*(\xi' J_{c'}) \subseteq \xi J_c$.
\end{lemma}
\begin{proof}
    Let $m' \in S_{\sigma'} - \ker c'$. Then
    \begin{align*}
        \langle \bar\phi^*(m'), c \rangle &= \langle m', \bar \phi_\kk(c) \rangle \\
        &= a^2 \langle m', c' \rangle,
    \end{align*}
    which is nonzero if $a \neq 0$. In this case, we have $\phi^*(J_c') \subseteq J_c$. If $a = 0$, then $\phi^*(\xi') = 0$ and the lemma still holds.
\end{proof}

\subsection{The general case}
Let $X$ be a quasinormal toric supervariety whose torus is $T=T_{N,c}$. Then $X_0 = X_\Sigma$ for a fan $\Sigma$ in $N_\RR$. By the results of the previous section, each cone $\sigma \in \Sigma$ is decorated by a suitable ideal $J^\sigma \supset J^\sigma_c$ of $\kk[S_\sigma]$ such that if $\tau$ is a face of $\sigma$, then $J^\tau$ is a particular localization of $J^\sigma$. We show that these data give a complete classification of quasinormal toric supervarieties for the given torus $T$.

\begin{proposition}
Let $\Sigma$ be a fan in $N_\RR$, $c \in N_\kk$, and $\{J^\sigma \subseteq \kk[S_\sigma] \mid \sigma \in \Sigma\}$ a collection of nonzero ideals such that each $J^\sigma \supseteq J^\sigma_c$. If $(J^\sigma)_{x^{m_{\sigma, \tau}}} = J^\tau$ whenever $\tau$ is a face of $\sigma$ in $\Sigma$, then the affine toric supervarieties $U_\sigma = \Spec (\kk[S_\sigma] \oplus \xi J^\sigma)$ can be glued together along their intersections to produce a toric supervariety whose underlying variety is $X_\Sigma$ and whose torus is $T_{N,c}$.
\end{proposition}
\begin{proof}
This follows from Proposition \ref{LocalizeProp}(a) and the classical version of this statement.    
\end{proof}

Hence, quasinormal toric supervarieties can be described by a fan $\Sigma$ decorated by suitable ideals. These ideals arrange nicely into a coherent ideal sheaf $\J \subseteq \O_{X_\Sigma}$, called the \textbf{fermionic sheaf}, obtained by gluing together the coherent sheaves $\tilde J^\sigma$ on each open affine described by a cone $\sigma \in \Sigma$. The fermionic sheaf of a quasinormal toric supervariety admits the following expected properties:
\begin{proposition}
    Let $\J$ be the fermionic sheaf of a quasinormal toric supervariety with underlying variety $X_\Sigma$ and torus $T = T_{N,c}$.
    \begin{enumerate}
        \item $\J$ contains the minimal fermionic sheaf $\J_c$, obtained by gluing the sheaves $\tilde J_c^\sigma$ for all $\sigma \in \Sigma$.
        \item $\J$ is $T_N$-equivariant.
    \end{enumerate}
\end{proposition}
\begin{proof}
\begin{enumerate}
    \item Note that $\J_c$ is well-defined by Proposition \ref{LocalizeProp}(b). Since each $J^\sigma \supseteq J^\sigma_c$, the claim follows.
    
    \item We know the claim holds for $U = U_\sigma$. The rest is standard; see e.g.\ \cite{Perling}. \qedhere
\end{enumerate}
\end{proof}

We have thus proven:

\begin{proposition}\label{ClassificationByIdealSheaves}
Let $X$ be a quasinormal supervariety and $T = T_{N,c}$. Then $X$ is a toric supervariety with supertorus $T$ if and only if $X \cong X_{\Sigma,\J} := (X_\Sigma, \O_{X_\Sigma} \oplus \xi \J)$ for a fan $\Sigma$ in $N_\RR$ and a nonzero $T_N$-equivariant ideal sheaf $\J \supseteq \J_c$ in $\O_{X_\Sigma}$.
\end{proposition}

\subsection{A combinatorial approach}

The classification data $(\Sigma,\J)$ above is only semi-combinatorial in the sense that $\J$ is an algebraic object. In this section, we represent $\J$ with more combinatorial data; namely, we replace $\J$ with a collection $\B$ of local generators of $\J$.

\begin{definition}
As in Definition \ref{AdmissibleDef}, for $m, m' \in S_\sigma$ we write $m' \leq_\sigma m$ if there is $m'' \in S_\sigma$ such that $m=m'+m''$.
\end{definition}

\begin{definition}\label{DecoratedFanDef}
Let $T_{N,c}$ be a supertorus and $\Sigma$ a fan in $N_\RR$. For each $\sigma \in \Sigma$, let $\B_\sigma$ be a nonempty finite set in $S_\sigma$. We say that $\B = \{\B_\sigma\}_{\sigma \in \Sigma}$ is a set of \textbf{$c$-admissible decorations} for the fan $\Sigma$ if for all cones $\sigma,\sigma' \in \Sigma$:
\begin{enumerate}[label=(\roman*)]
    \item For all $m \in S_\sigma - \ker c$, there is $b \in \B_\sigma$ such that $b \leq_\sigma m$.
    \item For all $b \in \B_\sigma$, there is $b' \in \B_{\sigma'}$ such that $b' \leq_{\sigma \cap \sigma'} b$.
    \item If $a,b \in \B_\sigma$ satisfy $a \leq_\sigma b$, then $a=b$.
\end{enumerate}
\end{definition}

Notice that unlike in Definition \ref{AdmissibleDef}, here we require minimality (iii) as part of admissibility. This requirement is not strictly necessary for the classification, but it simplifies some matters in Section \ref{FiberSection}.

\begin{lemma}
    Condition (ii) of Definition \ref{DecoratedFanDef} holds for all $\sigma, \sigma' \in \Sigma$ if and only if it holds for all $\sigma, \sigma'$ such that one is a face of the other.
\end{lemma}
\begin{proof}
    Suppose it holds in all cases of faces, and let $\tau = \sigma \cap \sigma'$. Now let $b \in \B_\sigma$. Then there is $a \in \B_\tau$ such that $a \leq_\tau b$, and likewise $b' \in \B_{\sigma'}$ such that $b' \leq_\tau a$ since $\tau$ is a face of both $\sigma$ and $\sigma'$. But then $b' \leq_\tau a \leq_\tau b$, so we are finished. 
\end{proof}

That is, we may replace condition (ii) with
\begin{enumerate}[label=(ii$'$)]
    \item Let $\tau$ be a face of $\sigma$. Then \begin{itemize}
        \item For all $b \in \B_\sigma$, there is $a \in \B_\tau$ such that $b \leq_\tau a$.
        \item For all $a \in \B_\tau$, there is $b \in \B_\sigma$ such that $a \leq_\tau b$.
    \end{itemize}
\end{enumerate}

The following lemma shows that $\B$ consists of local generating sets of the sheaf $\J$.

\begin{lemma}
    For each $\sigma \in \Sigma$, let $J^\sigma$ be a monomial ideal in $\kk[S_\sigma]$ and $\B_\sigma \subset S_\sigma$ a finite generating set of $J^\sigma$. Then
    \begin{enumerate}
        \item $J^\sigma \supseteq J^\sigma_c$ if and only if condition (i) of Definition \ref{DecoratedFanDef} holds.
        \item $J^\sigma_{x^{m_{\sigma, \tau}}} = J^{\tau}$ if and only if condition (ii) $\approx$ (ii$'$) of Definition \ref{DecoratedFanDef} holds.
        \item $\B_\sigma$ is minimal if and only if condition (iii) of Definition \ref{DecoratedFanDef} holds.
    \end{enumerate}
\end{lemma}

\begin{proof} First observe that for $m \in S_\sigma$, we have $x^m \in J^\sigma$ if and only if there is a generator $b \in \B_\sigma$ and an element $f \in \kk[S]_\sigma$ such that $f x^b = x^m$. In particular, $f$ will always be a monomial, so this is equivalent to $b \leq_\sigma m$.
\begin{enumerate}
    \item The ideal $J_c^\sigma$ is generated by $S_\sigma - \ker c$, so condition (i) exactly says that all generators of $J^\sigma_c$ belong to $J^\sigma$.
    \item Notice that $J^\tau$ and $J^\sigma_{x^{m_{\sigma,\tau}}}$ are both ideals of $\kk[S_\tau] = \kk[S_\sigma]_{x^{m_{\sigma, \tau}}}$, generated by $\B_\tau$ and $\B_\sigma$, respectively. Let $a \in \B_\tau$. By (ii'), there is $b \in \B_\sigma$ such that $b \leq_\tau a$, i.e.\ that $x^a \in J^\sigma_{x^{m_{\sigma, \tau}}}$. Likewise, let $b \in \B_\sigma$, so there is $a \in \B_\tau$ such that $a \leq_\tau b$, i.e.\ that $b \in J^\tau$. Therefore $J^\sigma_{x^{m_{\sigma, \tau}}} = J^{\tau}$. The converse follows by tracing backwards through the same reasoning.
    \item If $a \leq_\sigma b$ and $a \neq b$, then $b$ remains in the ideal generated by $\B_\sigma - \{b\}$, so $\B_\sigma$ is not minimal. Conversely, if $B_\sigma$ is not minimal, then we can find $a,b \in \B_\sigma$ satisfying $a \leq_\sigma b$. \qedhere
\end{enumerate}
\end{proof}

We write $X_{\Sigma,\B}$ for the toric supervariety defined by gluing the affine toric supervarieties $X_{\sigma, \B_\sigma}$ for $\sigma \in \Sigma$ along their intersections as before. Notice by Lemma \ref{UniqueGenSetLem(b)}(b) that each $\B_\sigma$ is unique up to invertible elements of $S_\sigma$.

\begin{corollary}
    Let $X$ be a quasinormal supervariety and $T = T_{N,c}$. Then $X$ is a toric supervariety with supertorus $T$ if and only if $X \cong X_{\Sigma,\B}$ for a fan $\Sigma$ with $c$-admissible decorations $\B$.
\end{corollary}

\begin{example}
Consider $(n|1)$-dimensional projective superspace as a supervariety; i.e.\ $\PP^{n|1} = (\PP^n, \Lambda^\bullet \O(-1)) = \Proj \kk[x_0, ..., x_n, \xi]$. Its underlying variety $\PP^n$ can be viewed as a toric variety $X_\Sigma$ where $\Sigma$ is the fan whose cones are the positive spans of the proper subsets of $\{e_1, ..., e_n, -(e_1+...+e_n)\}$.

Using the Proj construction, the affine charts of $\PP^{n|1}$ have coordinate rings $\kk[\frac{x_0}{x_i}, ..., \frac{x_n}{x_i}, \frac{\xi}{x_i}]$ for $i=0, ..., n$. Letting $z_i = \frac{x_i}{x_0}$ for $i=1, ..., n$ and $\eta = \frac{\xi}{x_0}$, we can write these rings as $\kk[z_1, ..., z_n, \eta]$ and $\kk[z_i^{-1}, z_i^{-1}z_1, ..., z_i^{-1} z_n, z_i^{-1} \eta]$ for $i=1, ..., n$. Notably, $\kk[z_1, ..., z_n]$ is the semigroup algebra of the dual cone to the $\RR_+$-span of $\{e_1, ..., e_n\}$, and $\kk[z_i^{-1}, z_i^{-1}z_1, ..., z_i^{-1} z_n]$ is that of $\{e_1, ..., \hat e_i, ..., e_n, -(e_1+...+e_n)\}$ where $\hat e_i$ denotes the exclusion of $e_i$. Denote these cones by $\sigma_0$ and $\sigma_i$ for $i=1, ..., n$ respectively.

The torus $T_{\ZZ^n}$ of this toric variety acts on the coordinate rings by $$(t_1, ..., t_n) \cdot x_0^{m_0} x_1^{m_1} \cdots x_n^{m_n} = x_0^{m_0} (t_1x_1)^{m_1} \cdots (t_nx_n)^{m_n}.$$ It follows that $\PP^{n|1}$ can be expressed as a decorated fan with $\sigma_0$ decorated by $\{1\}$ and $\sigma_i$ decorated by $\{x_i^{-1}\}$ (the characters by which $T_{\ZZ^n}$ acts on the respective odd generators in the coordinate rings). Then if $\PP^{n|1}$ is a toric supervariety, the torus must be $T_{\ZZ^n, c}$ where $\langle m, c\rangle = 0$ for all characters $x^m$ of the form $z_i^{-1} z_j$ for $i,j >0$. But $z_i^{-1} z_j = x_i^{-1} x_j$, so in particular we must have $c = 0$ or $c = (1, ..., 1)$, up to scale.
\end{example}

\subsection{Morphisms}

The category of quasinormal toric supervarieties with one odd dimension can then be described as having the above objects $X_{\Sigma, \B}$ with $T$ allowed to vary, and morphisms as in Definition \ref{MorphismDef}, which we characterize below. First, recall from Section \ref{NTVBackground} that a lattice map $\bar \phi : N \to N'$ is called \textbf{compatible with the fans} $\Sigma, \Sigma'$ in $N_\RR, N_\RR'$ if for each $\sigma \in \Sigma$, there is $\sigma' \in \Sigma'$ such that $\bar \phi_\RR(\sigma) \subseteq \sigma'$.

\begin{definition}
    Let $\phi_0 : X_\Sigma \to X_\Sigma'$ be a morphism of toric varieties induced by a lattice map $\bar \phi : N \to N'$ compatible with the fans $\Sigma$ in $N_\RR$ and $\Sigma'_\RR$ in $N'$. Let $\B$ and $\B'$ be $c$- and $c'$-admissible decorations for their respective fans. We say that $\bar \phi$ is \textbf{compatible with the decorations} if whenever $\bar \phi_\RR(\sigma) \subseteq \sigma'$ and $b' \in \B_{\sigma'}'$, there is $b \in \B_\sigma$ such that $b \leq_\sigma \bar \phi^*(b')$.
\end{definition}

\begin{proposition}\label{DecorationCompatibilityProp}
    Using the prior notation, let $\phi^\# : \O_{X_{\Sigma'}} \to \phi_*\O_{X_\Sigma}$ be the corresponding morphism of sheaves on $X_{\Sigma'}$, and let $Z \subset X_\Sigma$ and $Z' \subset X_{\Sigma'}$ be the closed subschemes corresponding to the ideal sheaves $\J$ and $\J'$ obtained from the decorations $\B$ and $\B'$, respectively. Then the following are equivalent:
    \begin{enumerate}
        \item $\bar \phi$ is compatible with the decorations.
        \item For each $\sigma \in \Sigma$ and $\sigma' \in \Sigma'$ such that $\bar\phi_\RR(\sigma) \subseteq \sigma'$, it holds that $\phi_0|_{U_\sigma}^*(J^{\sigma'}) \subseteq J^\sigma$.
        \item $\phi^\#|_{\J'}$ takes $\J'$ to $\phi_* \J$.
        \item $\phi|_Z$ is a morphism of schemes $Z \to Z'$.
    \end{enumerate}
\end{proposition}

\begin{proof}
    For (a) $\iff$ (b), observe that $\bar \phi^*(\B'_{\sigma'})$ generates $\phi_0|_{U_\sigma}^*(J^{\sigma'})$. Hence $\phi_0|_{U_\sigma}^*(J^{\sigma'}) \subseteq J^\sigma$ is equivalent to the condition that for all $b' \in \B'_{\sigma'}$, there is $b \in \B_\sigma$ such that $b \leq_\sigma \bar \phi^*(b')$.

    The map $\phi_0|_{U^\sigma}^* : \kk[S_{\sigma'}] \to \kk[S_\sigma]$ is the same as $$\Res^{\phi^{-1}(U_{\sigma'})}_{U_{\sigma}} \circ \phi^\#(U_{\sigma'}) : \O_{X_{\Sigma'}}(U_{\sigma'}) \to \O_{X_\Sigma}(U_\sigma),$$ so if we restrict further to $J^{\sigma'} \subseteq \kk[S_{\sigma'}]$ and $\J'(U_{\sigma'}) \subseteq \O_{X_{\Sigma'}}(U_{\sigma'})$, we see immediately that (b) $\iff$ (c). The fact that (c) $\iff$ (d) is standard.
\end{proof}

Hence if $X_{\sigma,J}$ and $X_{\sigma',J'}$ are affine, compatibility with decorations is exactly the condition that $\phi^*(J') \subseteq J$ as in Proposition \ref{ConeMorphismProp}. The following condition is equivalent to the condition that $\phi^*(\xi'J') \subseteq \xi J$.

\begin{definition}
Let $X_{\Sigma,\B}$ and $X_{\Sigma',\B'}$ be toric supervarieties with supertori $T_{N,c}$ and $T_{N',c'}$, respectively. If $(\bar\phi, a) : T_{N,c} \to T_{N'c,}$ is a morphism of supergroups, we say that $(\bar\phi, a)$ is \textbf{compatible with the decorated fans} if $\bar \phi : N \to N'$ is compatible with the fans, and either $a=0$, or $\bar\phi$ is compatible with the decorations.
\end{definition}

\begin{proposition}
    Let $N, N'$ be lattices and $\Sigma, \Sigma'$ fans in $N_\RR, N_\RR'$ with $c$-,$c'$-admissible decorations $\B$,$\B'$, respectively. Suppose $\phi = (\bar \phi, a) : T_{N,c} \to T_{N',c'}$ is a morphism of supergroups. Then $\phi$ extends to a morphism of toric supervarieties $X_{\Sigma, \B} \to X_{\Sigma', \B'}$ if and only if $(\bar \phi,a)$ is compatible with the decorated fans.
\end{proposition}
\begin{proof}
    First suppose that $\phi$ extends. By the analogous result in the purely even setting, $\bar \phi$ is compatible with the fans. Let $\sigma \in \Sigma$, so the extension $\phi$ restricts to a map $$\phi|_{U_{\sigma, J^\sigma}} : U_{\sigma, J^\sigma} \to U_{\sigma', J^{\sigma'}}$$ for any $\sigma' \in \Sigma'$ such that $\bar\phi_\RR(\sigma) \subseteq \sigma'$. By Proposition \ref{ConeMorphismProp}, we see that $a=0$ or $\phi_0|_{U_\sigma}^*(J^{\sigma'}) \subseteq J^\sigma$, which is sufficient by Proposition \ref{DecorationCompatibilityProp}.

    Conversely, suppose the compatibility conditions hold. By Proposition \ref{ConeMorphismProp}, these conditions guarantee that for each open affine toric subvariety $U_{\sigma, \B_\sigma} \subseteq X_{\Sigma, \B}$, there is an open affine toric subvariety $U_{\sigma', \B'_{\sigma'}} \subseteq X_{\Sigma', \B'}$ such that $\phi$ extends to a toric morphism $U_{\sigma, \B_\sigma} \to U_{\sigma', \B'_{\sigma'}}$. Moreover, these partial extensions are compatible with each other on intersections because intersections are also open affine toric subvarieties. Therefore these partial extensions may be glued together into a global extension, finishing our proof.
\end{proof}

\begin{definition}
    We define the category of \textbf{decorated fans} as follows: The objects are quadruples $(N,c,\Sigma,\B)$ where $N$ is a lattice, $c \in N_\kk$, $\Sigma$ is a fan in $N_\RR$, and $\B$ is a $c$-admissible family of decorations for $\Sigma$. The morphisms are tuples $(\bar\phi, a) : (N,c,\Sigma,\B) \to (N',c',\Sigma',\B')$ where $\bar\phi : N \to N'$ is a lattice map, $a \in \kk$ is such that $\bar \phi_\kk(c) = a^2 c'$, and $(\bar\phi, a)$ is compatible with the decorated fans. Composition of morphisms is given by $(\bar\phi', a') \circ (\bar \phi, a) = (\bar \phi' \circ \bar \phi, aa')$.
\end{definition}

\begin{theorem}
    The category of quasinormal toric supervarieties with one odd dimension is equivalent to the category of decorated fans.
\end{theorem}

\section{The geometry of quasinormal toric supervarieties}\label{GeometrySection}

Now that we have described the category of quasinormal toric supervarieties of one odd dimension, we seek to understand their geometry.

\subsection{Split toric supervarieties}\label{SplitSubsection}

Roughly speaking, a supervariety is split if it is smooth in the odd directions.

\begin{definition}
    A supervariety $X$ is \textbf{graded} if there is a coherent sheaf $\F$ of $\O_{X_0}$-modules such that $\O_X \cong \Lambda^\bullet \F$. It is called \textbf{split} if $\F$ can be chosen to be locally free (i.e.\ a vector bundle).
\end{definition}

As we have seen, every quasinormal $(n|1)$-dimensional toric supervariety is graded, with $\F=\J$ the fermionic sheaf and $\Lambda^2 \J = 0$.

\begin{proposition}\label{SplitProp}
    Let $X=X_{\Sigma,\B} = X_{\Sigma,\J}$ be a quasinormal toric variety with torus $T_{N,c}$. The following are equivalent:
    \begin{enumerate}
        \item $X$ is split.
        \item Each $\B_\sigma$ is a singleton.
        \item Each $J^\sigma$ is a principal ideal.
        \item $\J$ is a line bundle on $X_\Sigma$.
    \end{enumerate}
\end{proposition}

\begin{proof}
    We have (a) $\iff$ (d) from the definition of split and (b) $\iff$ (c) from the definition of principal. Now if (d) holds, then $\J|_{U_\sigma}$ is a line bundle on the affine toric variety $U_\sigma$. A result of Gubeladze \cite{Gubeladze} shows that an arbitrary vector bundle on an affine toric variety is trivial, so $\J|_{U_\sigma} = \tilde J^\sigma$ where $J^\sigma \cong \kk[S_\sigma]$ as a $\kk[S^\sigma]$-module. That is, $J^\sigma$ is principal for each $\sigma \in \Sigma$.

    Conversely, if each ideal is principal, then each $\tilde J^\sigma$ is a trivial line bundle, and so $\J$ trivializes over the cover of $X_\Sigma$ by the $U_\sigma$.
\end{proof}

\begin{definition}
    We shall henceforth refer to a family $\B$ of decorations as \textbf{split} if each $\B_\sigma$ is a singleton.
\end{definition}

\subsection{Smooth toric supervarieties}\label{SmoothSubsection}

We seek to establish a characterization of smooth toric supervarieties. First, we recall the definition of a smooth supervariety and a smooth cone.

\begin{definition}
    Let $X = (|X|,\O_X)$ be a supervariety and $\p \in X(\kk) = |X|$ a closed point. Denote by $B$ the stalk $\O_{X,\p}$. We say that $X$ is \textbf{smooth at} $\p$ if $\overline{B} := B/(B_1)$ is a regular local ring, and $B \cong \overline{B}[\xi_1, ..., \xi_n]$ for some $n \in \NN$. We say that $X$ is \textbf{smooth} if it is smooth at all points $\p$.
\end{definition}

It is shown in \cite{Sherman}, Proposition 3.7.4 that if $X$ is smooth at $\p$, then $B = \O_{X,\p}$ is an integral superdomain. Consequently, only integral supervarieties can be smooth, and the condition that $\overline{B}$ is a regular local ring is equivalent to the condition that $X_0 = \overline{X}$ is smooth at $\p$.

\begin{definition}
    Let $N$ be a lattice, and $\sigma \subset N_\RR$ a strongly convex rational polyhedral cone. We say that $\sigma$ is \textbf{smooth} if its minimal generators form part of a basis for $N$. A fan $\Sigma$ is \textbf{smooth} if all its cones are smooth.
\end{definition}

Let us begin in the affine case.

\begin{proposition}\label{SmoothAffineProp}
    Let $X$ be an affine toric supervariety with torus $T_{N,c}$. Then $X$ is smooth if and only if $X\cong X_{\sigma,\B_\sigma}$ for a smooth cone $\sigma$ and a singleton $\B_\sigma$.
\end{proposition}

\begin{proof}
    If $X$ is smooth, then so is $X_0$. But a smooth variety is normal, and so $X=X_{\sigma,\B_\sigma}$ for a smooth cone $\sigma$ and a $c$-admissible finite set $\B_\sigma = \{m_1, ..., m_s\} \subset S_\sigma$. Then the underlying variety is $X_0 \cong \Spec \kk[x_1, ..., x_r, x_{r+1}^{\pm 1}, ..., x_n^{\pm 1}] = \kk^r \times (\kk^\times)^{n-r}$ with torus $T_0 \cong (\kk^\times)^n$. Localizing at $\p = (0^r, 1^{n-r})$, we set $B=\kk[S]_\p$ and find $B \cong \overline{B}[x^{m_1} \xi, ..., x^{m_s} \xi]$ wherein $x^{m_i} \xi \notin \overline{B}[x^{m_j} \xi]$ for $i \neq j$. By definition of smoothness, it follows that $s=1$.

    On the other hand, if $\sigma$ is a smooth cone and $\B_\sigma$ is a singleton, then by the above reasoning we have $B \cong \overline{B}[x^m \xi] \cong \overline{B}[\xi]$ with $\overline{B}$ a regular local ring. Therefore $X_{\sigma,\B_\sigma}$ is smooth.
\end{proof}

The general case follows immediately from Propositions \ref{SplitProp} and \ref{SmoothAffineProp}.

\begin{theorem}
    Let $X$ be a toric supervariety with torus $T_{N,c}$. Then $X$ is smooth if and only if $X\cong X_{\Sigma,\B}$ for a smooth fan $\Sigma$ and split decorations $\B$.
\end{theorem}

\begin{example}
    In this example, we show that every $(1|1)$-dimensional toric supervariety is smooth. Note that the three 1-dimensional ordinary toric varieties (up to isomorphism) are the torus $\kk^\times$ and the affine and projective lines $\AA^1$ and $\PP^1$, all of which are smooth. Also, the two $(1|1)$-dimensional supertori (up to isomorphism) are $T_{\ZZ,0}$ and $T_{\ZZ,1} \cong Q(1)$.

    Observe that the only nonzero monomial ideal in $\kk[\kk^\times] \cong \kk[x^{\pm 1}]$ is the whole ring $(1)$. However, in $\kk[\AA^1] \cong \kk[x]$, we have $(x^n)$ for $n$ any nonnegative integer. Since all these ideals are principal and smoothness is affine-local, it indeed follows that every $(1|1)$-dimensional toric supervariety is smooth.

    We now consider the possibilities over $\PP^1$. Let $m,n\geq0$. If we glue together the affine toric supervarieties $\Spec \kk[x, x^m \xi]$ and $\Spec \kk[x^{-1}, x^{-n} \xi]$ along their intersection $\Spec \kk[x^{\pm 1}, \xi]$, then we obtain a toric supervariety whose underlying variety is $\PP^1$ and whose fermionic sheaf is isomorphic to the line bundle $\O(-m-n)$. In particular, if the supertorus is $Q(1)$, then $m,n\leq1$ and so there are three possibilities for the fermionic sheaf; namely $\O(0), \O(-1),$ and $\O(-2)$.
\end{example}

\begin{example}
Consider the $\Pi$-symmetric super-Grassmannian $\Gr(1|1,2|2)^\Pi$ of $(1|1)$-dimensional subspaces of $\kk^{2|2}$ that are invariant under parity switch $\Pi$. For $A$ a superalgebra, its $A$-points consist of the full-rank matrices $$\begin{pmatrix}
    x_1 & \xi_1 \\
    x_2 & \xi_2 \\
    \xi_1 & x_1 \\
    \xi_2 & x_2
\end{pmatrix}$$ modulo the right action by $Q(1)$, where such a matrix is said to be full-rank if it has an invertible $(1|1) \times (1|1)$ submatrix (i.e.\ either $x_1$ or $x_2$ is invertible in $A_0$).

It is seen in e.g.\ \cite{Onishchik}, Section 6 that $\Gr(1|1,2|2)^\Pi$ is covered by two charts with coordinates $$\begin{pmatrix}
    y & \eta \\
    1 & 0 \\
    \eta & y \\
    0 & 1
\end{pmatrix} \text{ and } \begin{pmatrix}
    1 & 0 \\
    z & \zeta \\
    0 & 1 \\
    \zeta & z
\end{pmatrix}$$ respectively, and with transition functions given by $z = y^{-1}$ and $\eta = - \zeta/z^2$. Using the prior example, we find that $\Gr(1|1,2|2)^\Pi$ has underlying variety $\PP^1$ and fermionic sheaf $\O(-2)$.

This space admits a natural left action by the subgroup of the general linear supergroup $\GL(2|2,\kk)$ that preserves $\Pi$; this subgroup is $$Q(2) = \left\{ \begin{pmatrix}
    B & C \\
    C & B
\end{pmatrix} \right\}$$ where $B$ and $C$ are $2 \times 2$ matrices consisting of even and odd elements, respectively. We observe that $Q(2)$ contains a copy of $Q(1)$ embedded as $$\begin{pmatrix}
    a & 0 & \tau & 0 \\
    0 & 1 & 0 & 0 \\
    \tau & 0 & a & 0 \\
    0 & 0 & 0 & 1
\end{pmatrix}$$
so that the action on the first coordinate chart
$$\begin{pmatrix}
    a & 0 & \tau & 0 \\
    0 & 1 & 0 & 0 \\
    \tau & 0 & a & 0 \\
    0 & 0 & 0 & 1
\end{pmatrix}\begin{pmatrix}
    y & \eta \\
    1 & 0 \\
    \eta & y \\
    0 & 1
\end{pmatrix}$$ corresponds to multiplication $$\begin{pmatrix}
    a & \tau \\
    \tau & a
\end{pmatrix} \begin{pmatrix}
    y & \eta \\
    \eta & y
\end{pmatrix}$$ in $Q(1)$. Therefore $\Gr(1|1,2|2)^\Pi$ is naturally a smooth toric supervariety with supertorus $Q(1)$.
\end{example}

\subsection{Fibers of $\J$}\label{FiberSection}

One way to glean information from the fermionic sheaf $\J$ is by computing its fibers, which will turn out to extract the fan decorations $\B$.

In the classical setting of a normal toric variety $X_\Sigma$ with torus $T_N$, the torus orbits correspond bijectively to cones $\sigma \in \Sigma$ (\cite{CLS}, Theorem 3.2.6).
We denote by $\Orb(\sigma)$ the $T_N$-orbit of $X_{\Sigma}$ corresponding to the cone $\sigma \in \Sigma$.

\begin{lemma}
    Let $\p$ be a closed point in $\Orb(\sigma)$. Then $\B_\sigma$ is unique when viewed as a subset of the character lattice of $\Stab_{T_N}(\p)$.
\end{lemma}
\begin{proof}
By Lemma \ref{UniqueGenSetLem(b)}(b), $\B_\sigma$ is unique up to multiplication by monomials which are invertible in $\kk[S_\sigma]$. The invertible monomials in $\kk[S_\sigma]$ are exactly $x^m$ for $m \in \sigma^\perp \cap M$, since $S_\sigma = \check \sigma \cap M$. Observe that $\sigma^\perp \cap M$ is the character lattice of the orbit $T_N \p$, viewed as an algebraic torus.

Let $b \in \B_\sigma$, and let $m \in \sigma^\perp \cap M$. Both $x^{b}$ and $x^{b+m}$ can be viewed as characters $T_N \to \kk^\times$. It suffices to show that they become the same character of $\Stab_{T_N}(\p)$ when we compose them with the map $\Stab_{T_N}(\p) \to T_N$. We have $T_N\p \cong T_N / \Stab_{T_N}(\p)$ by the orbit-stabilizer theorem, so the statement that $x^{m}$ belongs to the character lattice of $T_N\p$ is exactly the statement that $x^{m}$ is trivial on $\Stab_{T_N}(\p)$, as desired.
\end{proof}

\begin{proposition}
    Let $X_{\Sigma, \B}$ be a quasinormal toric supervariety with supertorus $T_{N,c}$, and let $\sigma \in \Sigma$. Let $\p$ be a closed point in $\Orb(\sigma)$. Then the fiber $\J(\p)$, as a representation of $\Stab_{T_N}(\p)$, is a direct sum of the characters in $\B_\sigma$.
\end{proposition}
\begin{proof}
    Localizing at $\p$ means we invert monomials not in $\p$. In the event that $\p$ belongs to the $T_N$-orbit corresponding to $\sigma$, we invert the monomials in $\sigma^\perp$. Hence, when we quotient $J^\sigma_\p$ by $\p J^\sigma_\p$, we are left with a direct sum of the desired characters.
\end{proof}

\begin{example}
    Consider the torus $T = T_{\ZZ^3, (1,0,0)} \cong Q(1) \times (\kk^\times)^2$ and the toric supervariety over $T$ with underlying variety $X_0 = \AA^3$ (from the cone $\sigma$ which is the positive orthant in $\RR^3$) and coordinate ring $$\kk[X] = \kk[x_1, x_2, x_3, x_1 \xi, x_2^{j_1} x_3^{k_1} \xi, ..., x_2^{j_\ell} x_3^{k_\ell} \xi]$$ for some fixed $\ell \in \NN$, $0 \leq j_1 \leq ... \leq j_\ell$, and $k_1 \geq ... \geq k_\ell \geq 0$. Observe that $\B_\sigma = \{x_1, x_2^{j_i} x_3^{k_i}\}_{i=1, ..., \ell}$. By Corollary \ref{ClassificationCor}, every toric supervariety with supertorus $T$ and underlying variety $X_0$ is of this form, since $J_{(1,0,0)} = (x_1)$. 
    
    Denote by $\tau_{ij}$ the two-dimensional cone generated by $e_i$ and $e_j$, and by $\rho_i$ the one-dimensional cone generated by $e_i$. One can verify the following computations, in which we conflate maximal ideals with classical points in $\AA^3$:
    \begin{itemize}
        \item For $\p \in \Orb(\sigma) = \{(0,0,0)\}$, the fiber is $$J(\p) \cong \kk x_1 \oplus \bigoplus_{i=1}^\ell \kk x_2^{j_i} x_3^{k_i}$$
        \item For $\p \in \Orb(\tau_{12}) = 0 \times 0 \times \kk^\times$, the fiber is $$J(\p) \cong \begin{cases}
            \kk x_1 \oplus \kk x_2^{j_1} &\text{ if } \ell>0 \text{ and } j_1>0 \\
            \kk 1 &\text{ if } \ell>0 \text{ and } j_1=0 \\
            \kk x_1 &\text{ if } \ell=0
        \end{cases}$$
        \item For $\p \in \Orb(\tau_{13}) = 0 \times \kk^\times \times 0$, the fiber is $$J(\p) \cong \begin{cases}
            \kk x_1 \oplus \kk x_3^{k_\ell} &\text{ if } \ell>0 \text{ and } k_\ell>0 \\
            \kk 1 &\text{ if } \ell>0 \text{ and } k_\ell=0 \\
            \kk x_1 &\text{ if } \ell=0
        \end{cases}$$
        \item For $\p \in \Orb(\rho_{1}) = 0 \times \kk^\times \times \kk^\times$, the fiber is $$J(\p) \cong \begin{cases}
            \kk 1 &\text{ if } \ell>0 \\
            \kk x_1 &\text{ if } \ell=0
        \end{cases}$$
        \item The fiber at any other closed point is $\kk 1$, since $x_1^{-1}$ is present in the local ring and $x_1 \xi \in \kk[X]$.
    \end{itemize}
\end{example}

\subsection{Orbits and stabilizers}
In the ordinary setting, both the orbit and the stabilizer of a closed point in a toric variety can be viewed as algebraic tori. In particular, if the torus is $T_N$ and $\p \in \Orb(\sigma)$, let $N_\sigma$ be the sublattice of $N$ spanned by the points in $\sigma \cap N$, and $N(\sigma) = N/N_\sigma$. Then $\Stab_{T_N}(\p) \cong T_{N_\sigma}$ and $T_N \cdot \p = \Orb(\sigma) \cong T_{N(\sigma)}$. Additionally, the Zariski closure of the orbit $\Orb(\sigma)$ is a toric variety with torus $T_{N(\sigma)}$. We will prove analogous results in the super case.

We recall from \cite{Fioresi} the notion of the stabilizer supergroup.
\begin{definition}
    Let $G$ be an affine algebraic supergroup acting on a supervariety $X$. Given a closed point $\p \in |X|$, the \textbf{stabilizer supergroup functor} is the group-valued functor $\Stab_G(\p)$ on the category of superalgebras, given by $$\Stab_G(\p)(A) = \{g \in G(A) \mid g \cdot \p_A = \p_A\}$$
    where $\p_A$ is $\p$, viewed as an $A$-point of $X$. By [2], $\Stab_G(\p)$ is naturally an affine algebraic supergroup.

    We may similarly define the \textbf{orbit functor} $G \cdot \p$ as $$(G \cdot \p)(A) = G(A) \cdot \p_A$$
    so that the orbit-stabilizer theorem holds, in the sense that $$(G \cdot \p)(A) \cong G(A) / \Stab_G(\p)(A)$$ at the level of sets.
\end{definition}

Let $\p \in |X|$ be a closed point in a toric supervariety $X$ with one odd dimension. To compute its orbit and stabilizer, we may assume without loss of generality that $X = Y_{\A,\B}$ is affine, so that $X = \Spec \kk[x_1, ..., x_r, \xi_1, ..., \xi_s] / J(Y_{\A,\B})$ by the proof of Lemma \ref{UniqueGenSetLem(b)}(c). Then $\p = (a_1, ..., a_r \mid 0, ..., 0)$ for some $a_i \in \kk$ on which 
 all the ordinary polynomials in $J(Y_{\A,\B})$ vanish. Suppose that $\kk[x_1, ..., x_r, \xi_1, ..., \xi_s] / J(Y_{\A,\B}) \cong \kk[x^{m_1}, ..., x^{m_r}, x^{n_1} \xi, ..., x^{n_s} \xi]$, and recall that $n_j = \sum_{i=1}^r p_i^j m_i$ for some $p_i^j \in \NN$, and $n_j = m_j$ for $j=1, ..., q$.

\begin{proposition}
    Using the above notation, it holds that
    $$\Stab_{T_{N,c}}(\p) \cong \begin{cases}
        T_{N_\sigma, c} &\text{ if } a^{p^j} = 0 \text{ for all } i=1, ..., s \\
        T_{N_\sigma} &\text{ otherwise}
    \end{cases}$$
    where $c$ lies in the subspace $(N_\sigma)_\kk \subset N_\kk$ and
    $$T_{N,c} \cdot \p \cong \begin{cases}
        T_{N(\sigma)} &\text{ if } a^{p^j} = 0 \text{ for all } i=1, ..., s \\
        T_{N(\sigma),\bar c} &\text{ otherwise}
    \end{cases}$$
    where $\bar c$ is the image of $c$ under $N \to N(\sigma)$.
\end{proposition}

\begin{proof}
    For $(t \mid \xi) \in T_{N,c}(A)$ and $(a_1, ..., a_r \mid \eta_1, ..., \eta_s) \in X(A)$, we use the group law from Proposition \ref{YA is a TSV Prop} to calculate
    \begin{align*}
        (t \mid \xi) &\cdot (a_1, ..., a_r \mid \eta_1, ..., \eta_s) \\
        = (&t^{m_1} a_1 + t^{m_1 }\langle m_1, c \rangle \xi \eta_1, ..., t^{m_q} a_1 + t^{m_q}\langle m_q, c \rangle \xi \eta_q, \\
        &t^{m_{q+1}} a_{q+1}, ..., t^{m_r} a_r \mid \\
        &t^{n_1} a_1 \xi + t^{n_1} \eta_1, ..., t^{n_q}a_q \xi + t^{n_q} \eta_q, \\
        &t^{n_{q+1}} a^{p^{q+1}} \xi + t^{n_{q+1}} \eta_{q+1}, ..., t^{n_{q+1}} a^{p^{q+1}} \xi + t^{n_{q+1}} \eta_{q+1})
    \end{align*}
    and so in particular,
    \begin{align*}
        (t \mid \xi) \cdot \p_A &= (t \mid \xi) \cdot (a_1, ..., a_r \mid 0, ..., 0) \\
        &= (t^{m_1} a_1, ..., t^{m_r} a_r \mid t^{n_1} a^{p^1} \xi, ..., t^{n_{s}} a^{p^{s}} \xi) \\
        &= (t^{m_1} a_1, ..., t^{m_r} a_r \mid t^{m_1} a_1 \xi, ..., t^{m_q} a_q \xi, t^{n_{q+1}} a^{p^{q+1}}, ..., t^{n_{s}} a^{p^{s}} \xi).
    \end{align*}
    It follows that $(t|\xi)$ stabilizes $\p_A$ if and only if $t \in \Stab_{T_N}(\p) \cong T_{N_\sigma}$ and either $a^{p^j} = 0$ for all $j=1, ..., s$, or $\xi=0$.
    
    In the former case, this implies that $a_j = 0$ for $j=1, ..., q$. Under the correspondence between closed points of $\Spec \kk[S_\sigma]$ and semigroup homomorphisms $S_\sigma \to \kk$ (\cite{CLS}, Proposition 1.3.1), the homomorphism corresponding to $\p$ sends $m_1, ..., m_q$ to 0. Therefore these $m_j$ are not invertible in $S_\sigma$, so in particular $m_j \notin \sigma^\perp$. In order to see that $c \in (N_\sigma)_\kk$, we must show that $\langle \sigma^\perp, c \rangle = 0$. But $\sigma^\perp$ must contained in the span of $\{m_{q+1}, ..., m_r\}$ because $m_1, ..., m_q \notin \sigma^\perp$. Since $\{m_{q+1}, ..., m_r\} \subseteq \ker c$, it indeed follows that $c \in (N_\sigma)_\kk$.

    Now observe that in either case, $\Stab_{T_{N,c}}(\p)(A)$ is a normal subgroup of $X(A)$, so the orbit identity also follows by taking the respective quotients.
\end{proof}

\begin{example}
    Let $T = T_{\ZZ^2, (1,0)}$. We consider the family of toric supervarieties $X^n := \Spec \kk[x_1, x_2, x_1 \xi, x_2^n \xi]$ for $n \geq 0$ and $X^\infty = \Spec \kk[x_1, x_2, x_1 \xi]$. Let $\p \in |X^n| = \AA^2$ be the closed point $(a,b)$. We calculate the stabilizers for such points in each of these toric supervarieties:
    \begin{center}
\begin{tabular}{|c || c | c | c|} 
\hline
\multicolumn{4}{|c|}{$\Stab_{T_{\ZZ^2, (1,0)}}(a,b)$} \\
\hline
 & $X^0$& $X^n$ for $n>0$ & $X^\infty$ \\ 
 \hline
 $a = b = 0$ & $T_{\ZZ^2}$ & $T_{\ZZ^2, (1,0)}$ & $T_{\ZZ^2, (1,0)}$ \\ 
 $a=0, b \neq 0$ & $T_{\ZZ \times 0}$ & $T_{\ZZ \times 0}$ & $T_{\ZZ\times 0, 1}$ \\
 $a \neq 0, b=0$ & $T_{0 \times \ZZ}$ & $T_{0 \times \ZZ}$ & $T_{0 \times \ZZ}$ \\
 $a,b \neq 0$ & $T_{0^2}$ & $T_{0^2}$ & $T_{0^2}$ \\
 \hline
\end{tabular}
\end{center}
The orbits can be calculated from these data using the orbit-stabilizer theorem.
\end{example}

If $X_{\Sigma,\B}$ is a quasinormal toric supervariety with torus $T_{N,c}$ and $\sigma \in \Sigma$, we denote by $\Orb_c(\sigma) = T_{N,c} \cdot \p$ the $T_{N,c}$-orbit of a closed point $\p \in \Orb(\sigma)$. Recall that the (Zariski) closure of $\Orb_c(\sigma)$ in $X_{\Sigma,\B}$ is the smallest closed subvariety of $X_{\Sigma,\B}$ containing $\Orb_c(\sigma)$. We denote this closure by $X_{\Sigma,\B}(\sigma)$.

\begin{proposition}
    Using the above notation, $X_{\Sigma,\B}(\sigma)$ is a toric supervariety with torus $\Orb_c(\sigma)$. We have two cases:
    \begin{enumerate}
        \item If $\Orb_c(\sigma) \cong T_{N(\sigma)}$ as in the first case of the previous proposition, then $X_{\Sigma,\B}(\sigma)$ is the ordinary toric variety with fan $\bar \Sigma = \{\bar \tau \mid \sigma \leq \tau \}$ where $\bar \tau$ denotes the image of $\tau$ in $N(\sigma)_\RR = N_\RR / (N_\sigma)_\RR$.

        \item If $\Orb_c(\sigma) \cong T_{N(\sigma), \bar c}$ as in the second case of the previous proposition, then $X_{\Sigma,\B}(\sigma)$ is the toric supervariety with fan $\bar \Sigma$ as above, and decorations $\B_{\bar \tau} = \B_\tau \cap \sigma^\perp$ for $\tau \geq \sigma$.
    \end{enumerate}
\end{proposition}
\begin{proof}
    Part (a) follows immediately from the classical situation (\cite{CLS}, Proposition 3.2.7). We obtain the fan in part (b) likewise. The fact that $X_{\Sigma,\B}(\sigma)$ is a toric supervariety follows immediately from the definitions. It therefore remains to find the decorations of part (b).

    Since $X_{\Sigma,\B}(\sigma)$ is a closed subvariety of $X_{\Sigma,\B}$, it is affine-locally a quotient at the level of rings. That is, if $U_{\bar \tau} \subseteq X_{\Sigma,\B}(\sigma)$ is the affine open corresponding to $\bar \tau$ for $\tau \geq \sigma$, then $\kk[U_{\bar \tau, \B_{\bar \tau}}]$ is a quotient of $\kk[U_{\tau, B_\tau}]$. In particular, $\kk[U_{\tau, \B_\tau}] \to \kk[U_{\bar \tau, \B_{\bar \tau}}]$ is the surjection $\kk[S_\tau, \xi \B_\tau] \to \kk[S_\tau \cap \sigma^\perp, \xi (\B_\tau \cap \sigma^\perp)]$ given by mapping any $x^m$ with $m \notin \sigma^\perp$ to 0. Therefore $\B_{\bar \tau}$ can be viewed as $\B_\tau \cap \sigma^\perp$, so we are finished.
\end{proof}

\subsection{Fiber products}\label{FiberProductSection}
In this section, we determine when fiber products (i.e.\ pullbacks) exist in the category of quasinormal toric supervarieties of one odd dimension. Consider the diagram
\begin{center}
\begin{tikzcd}[row sep=large, column sep=large]
& (N', c', \Sigma', \B') \arrow[d, "{(\bar \phi', a')}"] \\
(N'', c'', \Sigma'', \B'') \arrow[r, swap, "{(\bar \phi'', a'')}"]
& (Q, d, \Lambda, \C)
\end{tikzcd}
\end{center}
in the category of decorated fans, and suppose $a'$ and $a''$ are not both 0. We construct the fiber product as follows:

Define a lattice $$N := N' \times_Q N'' = \{(n', n'') \in N' \times N'' \mid \bar \phi'(n') = \bar \phi''(n'')\}$$ and a fan $$\Sigma := \Sigma' \times_\Lambda \Sigma'' = \{(\sigma', \sigma'') \in \Sigma' \times \Sigma'' \mid \text{there is } \lambda \in \Lambda \text{ such that } \bar \phi'_\RR(\sigma'), \bar \phi''_\RR(\sigma'') \subseteq \lambda\}$$ so that $(N,\Sigma)$ is the fiber product of
\begin{center}
\begin{tikzcd}[row sep=large, column sep=large]
& (N', \Sigma') \arrow[d, "{\bar \phi'}"] \\
(N'', \Sigma'') \arrow[r, swap, "{\bar \phi''}"]
& (Q, \Lambda)
\end{tikzcd}
\end{center}
in the usual category of fans (or equivalently, toric varieties). Define $$c = c' \times_d c''  := ((a'')^2c',(a')^2c'')$$
and observe that
\begin{align*}
    \bar\phi'((a'')^2 c') &= (a' a'')^2 d \\
    &= \bar\phi''((a')^2 c''),
\end{align*}
so $c \in N_\kk$.

It remains to describe the decorations. Towards this end, we first describe the character lattice $M=N^*$. We also write $M'=(N')^*$, $M''=(N'')^*$, and $P=Q^*$. 
\begin{definition}
    Let $S$ be a semigroup contained in a lattice $L$. The \textbf{saturation} of $S$ in $L$ is given by $\Sat(S) = \{\ell \in L \mid k\ell \in S \text{ for some } k \in \NN\}$.
\end{definition}

We collect some elementary results about saturations of semigroups:

\begin{lemma}\label{SaturationLemma}
\begin{enumerate}
    \item If $S \subseteq L$ is a semigroup, then $\Sat(S)$ is the smallest saturated semigroup in $L$ that contains $S$.
    
    \item Let $L'$ be a sublattice of a lattice $L$. Then $L/L'$ is also a lattice if and only if $L'$ is saturated in $L$.

    \item The saturation of an affine semigroup is affine.
\end{enumerate}
\end{lemma}

We may now apply this lemma to describe $M$.

\begin{lemma}
\begin{enumerate}
    \item The natural surjection $\pi : M' \times M'' \to M$ has kernel $\Sat\big( (\bar \phi')^* \times (-\bar \phi'')^*(P) \big)$, so that $M \cong (M' \times M'')/ \Sat\big( (\bar \phi')^* \times (-\bar \phi'')^*(P) \big)$.

    \item If $\sigma = (\sigma', \sigma'') \in \Sigma$, then $S_\sigma$ is the saturation of $\pi(S_{\sigma'} \times S_{\sigma''})$.
\end{enumerate}
\end{lemma}

\begin{proof}
\begin{enumerate}
    \item We have a surjection $\pi : M' \times M'' \to M$ given by $\langle \pi(m', m''), (n', n'') \rangle = \langle m', n' \rangle + \langle m'', n'' \rangle$ for $(m',m'') \in M' \times M''$ and $(n',n'') \in N$. Notice that
    \begin{align*}
        &\langle \pi(m' + (\bar \phi')^*(p), m'' - (\bar \phi'')^*(p)), (n', n'')\rangle \\
        &= \langle m', n' \rangle + \langle m'', n'' \rangle + \langle (\bar \phi')^*(p), n' \rangle - \langle (\bar \phi')^*(p), n' \rangle \\
        &= \langle m', n' \rangle + \langle m'', n'' \rangle + \langle p, \bar \phi'(n') \rangle - \langle p, \bar \phi''(n'') \rangle \\
        &= \langle m', n' \rangle + \langle m'', n'' \rangle
    \end{align*}
    since $\phi'(n') = \phi''(n'')$, so $\ker \pi$ contains $(\bar \phi')^* \times (-\bar \phi'')^*(P)$. But $M$ is a lattice, so by Lemma \ref{SaturationLemma}(b) we know $\ker\pi$ is saturated and hence by (a), $\ker \pi \supseteq \Sat\big( (\bar \phi')^* \times (-\bar \phi'')^*(P) \big)$. Equality follows by counting dimensions.

    \item We first note that $S_\sigma \supseteq \pi(S_{\sigma'} \times S_{\sigma''})$ because if $(m', m'') \in S_{\sigma'} \times S_{\sigma''}$, then both summands of $\langle \pi(m', m''), (n', n'') \rangle = \langle m', n' \rangle + \langle m'', n'' \rangle$ are nonnegative for all $(n', n'') \in \sigma$. Then $\Sat(\pi(S_{\sigma'} \times S_{\sigma''})) \subseteq S_\sigma$ is a saturated semigroup which is affine by Lemma \ref{SaturationLemma}(c). Hence by the universal property of the fiber product, the following diagram commutes:
    \begin{center}
    \begin{tikzcd}[row sep=large, column sep=large]
        \kk[\Sat(\pi (S_{\sigma'} \times S_{\sigma''}))] \arrow[dr, bend left]
         & & \\
        & \kk[S_\sigma] \arrow[ul, bend left, dashed] 
        & \kk[S_{\sigma'}] \arrow[l]
            \arrow[ull, bend right]  \\
        & \kk[S_{\sigma''}] \arrow[u]
            \arrow[uul, bend left]
        & \kk[S_\lambda] \arrow[l, "(\phi'')^*"] \arrow[u, swap, "{(\phi')^*}"]
    \end{tikzcd}
    \end{center}
    In this diagram, the dashed arrow arises from the universal property of the fiber product in the (opposite) category of toric varieties, and the unlabelled solid arrows arise from the natural morphisms of semigroups. The composition of the dashed arrow with the arrow in the reverse direction is the unique endomorphism of $\kk[S_\sigma]$ satisfying the universal property. Since $\Sat(\pi(S_{\sigma'} \times S_{\sigma''})) \subseteq S_\sigma$, these arrows are isomorphisms. Hence $S_\sigma = \Sat(\pi (S_{\sigma'} \times S_{\sigma''}))$. \qedhere
\end{enumerate}
\end{proof}

We now describe the decorations. For each $\sigma = (\sigma', \sigma'') \in \Sigma$, let 
$$J^\sigma = \begin{cases}
    J_c^\sigma + ((\B'_{\sigma'} \times \{0\}) \cup (\{0\} \times \B_{\sigma''}'')) &\text{ if } a',a'' \neq 0 \\
    J_c^\sigma + (\B'_{\sigma'} \times \{0\}) &\text{ if } a'=0 \\
    J_c^\sigma + (\{0\} \times \B_{\sigma''}'') &\text{ if } a''=0
\end{cases}$$
and define $\B_\sigma = \B'_{\sigma'} \times_\C \B''_{\sigma''}$ as the minimal generating set of $J^\sigma$. We write $\B' \times_\C \B''$ for the decorations $\{\B'_{\sigma'} \times_\C \B''_{\sigma''}\}_{(\sigma', \sigma'') \in \Sigma}$.
\begin{lemma}
    The collection $\B = \B' \times_\C \B''$ defined above is $c$-admissible.
\end{lemma}
\begin{proof}
    It suffices to show that condition (ii$'$) holds. Let $\tau = (\tau', \tau'')$ be a face of $\sigma = (\sigma', \sigma'')$. Let $m \in J^\sigma$, so either $m \in J^\sigma_c$ or $m$ belongs to the ideal generated by $\B'_{\sigma'} \times \{0\}$ or $\{0\} \times \B_{\sigma''}''$. In the former case, condition (ii$'$) holds because $(J^\sigma_c)_{x^{m_{\sigma, \tau}}} = J^\tau_c$. In the latter case, we apply condition (ii$'$) of $\B'$ or $\B''$ to arrive at the result. The other half of the argument follows by the same reasoning.
\end{proof}

\begin{lemma}
    Let $\bar\psi' : N \to N'$ and $\bar\psi'' : N \to N''$ be the natural maps. Then $(\bar\psi', a'')$ and $(\bar\psi'', a')$ are compatible with the decorated fans.
\end{lemma}
\begin{proof}
Let $\sigma = (\sigma', \sigma'') \in \Sigma$. It suffices to show that $(\psi')^*(\xi' J^{\sigma'}) \subseteq \xi J^\sigma$ and $(\psi'')^*(\xi'' J^{\sigma''}) \subseteq \xi J^\sigma$. But the images of $J^{\sigma'}$ and $J^{\sigma''}$ are exactly the ideals generated by $\B'_{\sigma'} \times \{0\}$ and $\{0\} \times \B_{\sigma''}''$, respectively. These generators are present in $J^\sigma$ exactly when $a' \neq 0$ and $a'' \neq 0$, respectively, so the lemma follows.
\end{proof}

The above lemmas provide the data sufficient to describe the fiber product.

\begin{proposition}
    If $a'$ and $a''$ are not both 0, the fiber product is given by
    \begin{center}
    \begin{tikzcd}[row sep=large, column sep=large]
    (N, c, \Sigma, \B) \arrow[r, "{(\bar\psi', a'')}"] \arrow[d, swap, "{(\bar\psi'', a')}"] & (N', c', \Sigma', \B') \arrow[d, "{(\bar \phi', a')}"] \\
    (N'', c'', \Sigma'', \B'') \arrow[r, swap, "{(\bar \phi'', a'')}"]
    & (Q, d, \Lambda, \C)
    \end{tikzcd}
    \end{center}
    where
    \begin{align*}
        N &= N' \times_Q N'' \\
        c &= c' \times_d c'' \\
        \Sigma &= \Sigma' \times_\Lambda \Sigma'' \\
        \B &= \B' \times_\C \B'' \\
        \bar\psi' &: N \to N' \text{ and } \bar\psi'' : N \to N'' \text{ are the natural maps}.
    \end{align*}
\end{proposition}
\begin{proof}
    We have already seen that the diagram commutes in the category of decorated fans. Suppose we have a commutative diagram
    \begin{center}
    \begin{tikzcd}
        (\tilde N, \tilde c, \tilde \Sigma, \tilde \B)
        \arrow[drr, bend left,"{(\tilde \psi', \tilde a'')}"]
        \arrow[ddr, bend right,swap, "{(\tilde \psi'', \tilde a')}"] & & \\
        & (N,c,\Sigma,\B) \arrow[r, "{(\bar\psi', a'')}"] \arrow[d, swap, "{(\bar\psi'', a')}"]
        & (N', c', \Sigma', \B') \arrow[d, "{(\bar\phi', a')}"] \\
        & (N'', c'', \Sigma'', \B'') \arrow[r, swap, "{(\bar\phi'', a'')}"]
        & (Q,d,\Lambda,\C)
    \end{tikzcd}
    \end{center}
    so there is a unique map $\tilde\varphi : \tilde N \to N$ compatible with the fans $\tilde \Sigma$ and $\Sigma$.
    
    If $a' \neq 0$, then define $\tilde a = \frac{\tilde a'}{a'}$. Likewise if $a'' \neq 0$, we define $\tilde a = \frac{\tilde a''}{a''}$. If both are nonzero, we observe that $a' \tilde a'' = \tilde a' a''$ and hence $\tilde a$ is well-defined, and unique with respect to the property that $\tilde a a' = \tilde a'$ and $\tilde a a'' = \tilde a''$. Thus if $$(\tilde \varphi, \tilde a) : (\tilde N, \tilde c, \tilde \Sigma, \tilde \B) \to (N,c,\Sigma,\B)$$ is a morphism of decorated fans, it satisfies the universal property of the fiber product in the above diagram. We have
    \begin{align*}
        \tilde \varphi_\kk(\tilde c) &= ((\tilde a'')^2 c', (\tilde a')^2, c'') \\
        &= \tilde a^2 ((a'')^2 c', (a')^2, c'') \\
        &= \tilde a^2 c,
    \end{align*}
    so it remains to show that if $\tilde a \neq 0$, then $\tilde \varphi$ is compatible with the decorations.

    Suppose $\tilde \varphi_\RR(\tilde \sigma) \subseteq \sigma$. Looking at the affine patches defined by $\tilde \sigma$ and $\sigma = (\sigma', \sigma'')$, we obtain the diagram
    \begin{center}
    \begin{tikzcd}[row sep=large, column sep=large]
        \kk[S_{\tilde \sigma}, \tilde \xi \tilde \B_{\tilde \sigma}]
         & & \\
        & \kk[S_\sigma, \xi \B_\sigma]
        & \kk[S_{\sigma'}, \xi' \B'_{\sigma'}] \arrow[l, swap, "{(\psi')^*}"]
            \arrow[ull, swap, bend right, "{(\tilde \psi')^*}"]  \\
        & \kk[S_{\sigma''}, \xi'' \B''_{\sigma''}] \arrow[u, "{(\psi'')^*}"]
            \arrow[uul, bend left, "{(\tilde \psi')^*}"]
        & \kk[S_\lambda, \eta \C_\lambda] \arrow[l, "(\phi'')^*"] \arrow[u, swap, "{(\phi')^*}"]
    \end{tikzcd}
    \end{center}

    Let $m \in J^\sigma$, where $\tilde \varphi_\RR(\tilde \sigma) \subseteq \sigma$. Then either $m \in J^\sigma_c$, or $m \in (\B'_{\sigma'} \times \{0\})$ and $a''\neq 0$, or $m \in (\{0\} \times \B''_{\sigma''})$ and $a' \neq 0$. If $m \in J_c^\sigma$, then $\varphi^*(m) \in J_{\tilde c}^{\tilde \sigma}$ by Lemma \ref{JcPreservedLemma}. Otherwise we may assume without loss of generality that $m \in (\B'_{\sigma'} \times \{0\})$ and $a'' \neq 0$. Then there is $b \in \B'_{\sigma'}$ such that $b \leq_\sigma m$. It follows that $(\tilde \psi')^*(x^b \xi')$ is nonzero in $\kk[S_{\tilde \sigma}, \tilde \xi \tilde \B_{\tilde \sigma}]$ since $a'' \neq 0$. Hence $(\tilde \psi')^*(b) \leq_{\tilde \sigma} \tilde \varphi^*(m)$ and $(\tilde \psi')^*(b) \in J^{\tilde \sigma}$, so indeed $\tilde \varphi^*(m) \in J^{\tilde \sigma}$, as desired. Therefore $\tilde \varphi$ is compatible with the decorations, and we are finished.
\end{proof}

\begin{remark}
    If $a'=a''=0$, then the fiber product does not exist in general. In particular, commutativity of the diagram 
    \begin{center}
    \begin{tikzcd}
        (\tilde N, \tilde c, \tilde \Sigma, \tilde \B)
        \arrow[drr, bend left,"{(\tilde \psi', \tilde b')}"]
        \arrow[ddr, bend right,swap, "{(\tilde \psi'', \tilde b'')}"] & & \\
        & (N,c,\Sigma,\B) \arrow[r, "{(\bar\psi', b')}"] \arrow[d, swap, "{(\bar\psi'', b'')}"]
        & (N', c', \Sigma', \B') \arrow[d, "{(\bar\phi', 0)}"] \\
        & (N'', c'', \Sigma'', \B'') \arrow[r, swap, "{(\bar\phi'', 0)}"]
        & (Q,d,\Lambda,\C)
    \end{tikzcd}
    \end{center}
    does not guarantee the existence of $\tilde b$ such that $\tilde b' = \tilde b b'$ and $\tilde b'' = \tilde b b''$. 
\end{remark}

\bibliographystyle{hsiam}
\bibliography{refs}

\end{document}